\documentclass[11pt]{amsart}
\usepackage{amsmath,amssymb}
\usepackage{graphicx}
\newcommand\D{{\not\! d}}
\newcommand\E{{\not\! \bf e}}

\newcommand\ee{{\mathbf e}}

\newcommand\mm{{\mathbf m}}
\newcommand\nn{{\mathbf n}}
\newcommand\kk{{\mathbf k}}
\newcommand\pp{{\mathbf p}}
\newcommand\rr{{\mathbf r}}
\newcommand\CC{\mathbb C}

\newcommand\RR{\mathbb R}
\newcommand\ZZ{\mathbb Z}
\newcommand\NN{\mathbb N}

\newcommand\A{{\mathcal A}}

\newcommand\GL{{\mathrm{GL}}}

\newcommand\SU{{\mathrm{SU}}}
\newcommand\U{{\mathrm{U}}}

\newcommand\End{\operatorname{End}}

\makeatletter 
\newcommand\matc[4]{\left( {#1\@@atop #3}{#2\@@atop #4}\right)}
\newcommand\matr[4]{\left( {\hfill #1\@@atop\hfill #3}{\hfill
#2\@@atop\hfill #4}\right)}
\newcommand\matl[4]{\left( { #1\@@atop #3}{ #2\@@atop\hfill #4}\right)}
\makeatother
\newcommand\widearray[1]{\renewcommand\arraystretch{1.4} \begin{array}{#1}}

\theoremstyle{plain}
\newtheorem{thm}{Theorem}[section]

\newtheorem{prop}[thm]{Proposition}

\theoremstyle{definition}

\theoremstyle{remark}
\newtheorem*{remark}{Remark}

\title[Two stochastic models of a random walk]{Two stochastic models of a
random walk in the $\U(n)$-spherical duals of $\U(n+1)$}

\author{F. A. Gr\"unbaum}
\author{I. Pacharoni}
\author{J. Tirao}
\address{Department of Math. University of California,
Berkeley, California 94705} \email{grunbaum@math.berkeley.edu}
\address{CIEM-FaMAF, Universidad Nacional de C\'ordoba, 5000 C\'ordoba, Argentina}
\email{pacharon@mate.uncor.edu, tirao@mate.uncor.edu}
\date{\today}

\begin{document}

\begin{abstract}
The random walk to be considered takes place in the
$\delta$-spherical dual of the group $\U(n+1)$, for a fixed finite
dimensional irreducible representation $\delta$ of  $\U(n)$. The
transition matrix comes from the three term recursion relation
satisfied by a sequence of matrix valued orthogonal polynomials
built up from the irreducible spherical functions of type $\delta$
of $\SU(n+1)$. One of the stochastic models is an urn model and the
other is a Young diagram model.
\end{abstract}
\maketitle

\section{Introduction}

Around 1770 D. Bernoulli studied a model for the exchange of heat between two bodies.
This model can also be seen as a description of the diffusion of a pair of
incompressible gases between two containers. This model was independently analyzed
by S. Laplace around 1810, see the references in \cite{F}. Another model of similar
characteristics was introduced by P. and T. Ehrenfest in 1907 in connection with the
controversies surrounding the work of L. Boltzmann in the kinetic theory of gases
dealing with
reversibility and convergence to equilibrium. Boltzmann had apparentlly
deduced his H-theorem dictating convergence
to equilibrium starting from the time reversible equations of Newton. For a nice
account of this see \cite{K}. Both of these models are instances of discrete time Markov
chains with fairly explicit tridiagonal one-step transition probability matrices which are
obtained by considering carefully the underlying stochastic mechanism that connects
the state of the system at two consecutive values of time.

The second model features two urns, I and II, that share a total of N balls. The state
of the system at time $n$ is the number of balls in urn I. Each ball has a different label from
the set $1,2,...,N$. At time $n$ a number $j$ in the set $1,2,...,N$ is chosen with
equal probabilities and the ball with this label is moved from the urn where it sits
to the other urn. This gives the state of the system at time $n+1$. Writing down the
one-step transition probability matrix is now a matter of counting carefully.

While it had been possible to obtain interesting answers for these
two models for quite some time, it is only much more recently that
some very nice connections have been noticed between these models
and some basic sets of discrete orthogonal polynomials, namely the
Krawtchouk and the dual Hahn polynomials. Moreover although there
are many ways of arriving at these polynomials it is relevant to
mention here that they can be realized as the ''spherical functions"
for certain finite bihomogeneous spaces. A very good reference for
this material is \cite{S}. We stress the remarkable fact that these
two models of old vintage and clear physical significance can be solved
in terms of the simplest of all hypergeometric functions, namely
$_2F_1$ and $_3F_2$.

As many readers certainly know many of the classical special
functions of mathematical physics, such as the Legendre, the Hermite
and the Laguerre polynomials, could have been obtained for the first
time as spherical functions for certain symmetric spaces. A good
basic reference here is \cite{V}. The way that things developed
historically is, of course, completely different.

The interplay between important physical problems and certain tools
that arise naturally in group representation theory constitutes the
theme of this paper. The situation described here is the reverse of
what has been discussed above for the Bernoulli-Laplace and the
Ehrenfest models: we will go from group representation theory to
some concrete models that might be of some physical interest. We
will start from a matrix that is obtained from group representation
theory and try to build a model that goes along with it. The models
constructed here are certainly not the only possible ones. More
natural ones might be lurking around.

In a series of papers including \cite{T1,T2, GPT, GPT1, GPT2, GPT3,
P, PT1, PT2, PT3, PT4} one considers matrix valued spherical
functions associated to a pair $(G,K)$ arriving at sequences of
matrix valued polynomials of one real variable satisfying a three
term recursion relation whose semi-infinite block tridiagonal matrix
is stochastic, i.e. the entries are non-negative and the sum of the
elements in any row is $1$. This matrix depends on a number of free
parameters that have a very definite group theoretical meaning. The
important point is that the tools developed in the papers just
mentioned allow one to give explicit expressions, in terms of some
definite integrals, of all the entries of any power of the original
matrix. This means that if one could think of a nice Markov chain
with this matrix as its one-step transition probability matrix one
would have an explicit form for the entries of the $n$-step
transition probability matrix. Many readers will recognize that this
is exactly what S. Karlin and J. McGregor, see \cite{KMcG}, proposed
as a way of exploiting orthogonal polynomials and the role they play
in the spectral analysis of certain finite or semi-infinite
tridiagonal matrices. The method advocated in \cite{KMcG} starts
with a so called birth-and-death process whose one-step tridiagonal
transition matrix is easily constructed from the given model and one
has to look for the corresponding spectral information: the
eigenfunctions and the spectral measure. Here we travel this road in
the opposite direction in a more elaborate set-up.

\bigskip

The relation between matrix valued orthogonal polynomials, block
tridiagonal matrices and Quasi-Birth and Death processes has been
first exploited independently in \cite{DRSZ,G} as well as in later
papers by these authors.

\smallskip

We will consider several random walks whose configuration spaces are
subsets of $\hat\U(n+1)(\kk)$, the so call $\kk$-spherical dual of $\U(n+1)$, and
whose one-step transition matrices come from the stochastic matrix
that appears in   \cite{PT2} and \cite{P}, see also \cite{PT4}. The
dual of $\U(n+1)$ is the set $\hat\U(n+1)$ of all equivalence classes of finite
dimensional irreducible representations of $\U(n+1)$. These equivalence
classes are parametrized by the $n+1$-tuples of integers
$\mm=(m_1,\dots,m_{n+1})$ subject to the conditions $m_1\ge
\cdots\ge m_{n+1}$.

If $\kk=(k_1,\cdots,k_n)\in\hat \U(n)$, the $\kk$-spherical
dual of $\U(n+1)$ is the subset  $\hat\U(n+1)(\kk)$ of $\hat\U(n+1)$ of the
representations of $\U(n+1)$ whose restriction to $\U(n)$ contains the
representation $\kk$. Then it is well known, see \cite{V}, that
$\hat\U(n+1)(\kk)$ corresponds to the set of all $\mm$'s as above that satisfy the
extra constraints
\begin{equation}\label{constrains}
m_i\ge k_i\ge m_{i+1},\quad \text{ for all}\quad i=1,\dots,n.
\end{equation}

In other words $\hat\U(n+1)(\kk)$ can be visualized as the subset of
all points $\mm$ of the integral lattice $\ZZ^{n+1}$ in the set
$$[k_1,\infty)\times[k_2,k_1]\times\cdots\times[k_{n-1},k_n]\times(-\infty,k_n].$$
An example is given in the figure below.

\begin{figure}[h]
\centering
\includegraphics[width=0.4\textwidth]{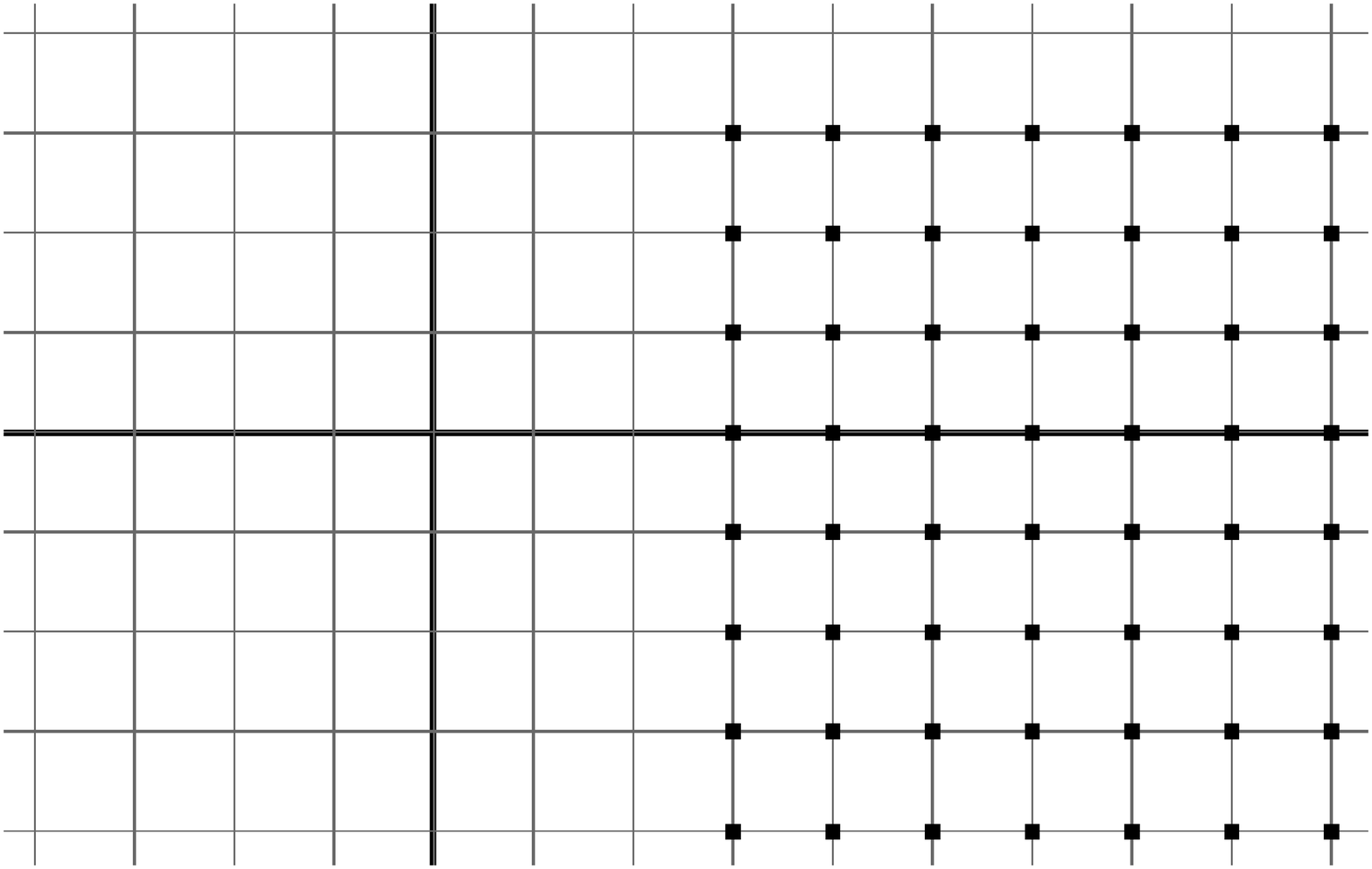}
\caption{$\hat\U(n+1)(\kk)$,\;$n=1$,\;$k_1=3$.}
\end{figure}

\

We can now state more precisely the point of this paper: starting
from the stochastic matrix $M$ that appears in \cite{PT2} and
\cite{P}, we describe a random mechanism that gives rise to a Markov
chain whose state space is the subset of $\hat\U(n+1)(\kk)$ of all
$\mm\in\hat\U(n+1)(\kk)$ such that $s_{\mm}=s_{\kk}$ and  $k_n\ge0$
($s_{\mm}=m_1+\cdots+m_{n+1}, s_{\kk}=k_1+\cdots+k_n$), and whose
one-step transition matrix coincides with the one we started from.
The construction in \cite{GPT} and \cite{PT2} deals with the case of
$(\SU(3),\U(2))$ but in \cite{PT3} and \cite{P} this was extended to
the case of $(\SU(n+1),\U(n))$.

One step of the Markov evolution will consist of two substeps taken
in succesion. In the first substep one of the values of $m_i$
increases by one, subject to the constraints (1). In the second
substep one of the new values of our $m_i$'s decreases by one, again
this is subject to the same constraints. Thus from the configuration
$\mm$ one could for instance go to $\mm-\ee_i+\ee_j$ or one could
stay put at $\mm$.  We use the notation $\ee_i$ for the vector with
its $i$th component equal to $1$ and all the others equal to $0$.
Any state has a total of at most $n(n+1)+1$ positions where it can
move in one complete step of our process consisting of two simpler
steps. Keep in mind that the two succesive simpler steps can end up
with our random walker in the initial state. We will analyze in
detail the simpler substeps that constitute one full step of our
process. This will take up most of the analysis in the next
sections.

We now describe the contents of the paper.

In Section \ref{esfericas} we collect the necessary material to
state and explain a three term recursion relation (with matrix coefficients) for a sequence of
matrix valued orthogonal polynomials, built up from irreducible
spherical functions of a fixed type associated to the pair
$(\SU(n+1), \U(n))$.
This should help the reader make the connection
between \cite{PT2, P} and the present paper.

In Section \ref{factorizacion} we construct a factorization of the
stochastic matrix that define the three term recursion relation for
the sequence of matrix valued orthogonal polynomials given in the
previous section. This factorization into two stochastic matrices leads to the two substeps
mentioned above.

Before starting the analysis of our general urn model in Section
\ref{urnmodelUn} for one of the substeps, we describe in detail in  Section
\ref{urnmodelU3}  an urn model for
$n=2$.

The definition of the stochastic matrix $M$ alluded above, as well
as its factorization make sense for any $\mm\in\hat\U(n+1)(\kk)$.

To each  configuration $m_1\ge m_2\ge \cdots\ge m_n\ge0$ of $n$
integer numbers we associate its Young diagram, a combinatorial
object which has $m_1$ boxes in the first row, $m_2$ boxes in the
second row, and so on down to the last row which has $m_n$ boxes.
For example the Young diagram associated to the configuration
$6\ge4\ge4\ge3$ is

\begin{figure}[h]
\centering
$$\begin{array} {l}
\begin{array}{|c|c|c|c|c|c|}
\hline \makebox[1.5mm]{}&\makebox[1.5mm]{} &\makebox[1.5mm]{}&\makebox[1.5mm]{}&\makebox[1.5mm]{}&\makebox[1.5mm]{}\\ \hline \end{array}\\
\begin{array}{|c|c|c|c|} \makebox[1.5mm]{}& \makebox[1.5mm]{} & \makebox[1.5mm]{}&\makebox[1.5mm]{}\\ \hline \end{array}\\
\begin{array}{|c|c|c|c|} \makebox[1.5mm]{}&\makebox[1.5mm]{}&\makebox[1.5mm]{}&\makebox[1.5mm]{}\\ \hline \end{array}\\
\begin{array}{|c|c|c|} \makebox[1.5mm]{}&\makebox[1.5mm]{}&\makebox[1.5mm]{}\\ \hline \end{array}\\
\end{array}$$
\caption{}
\end{figure}

Young diagrams and their relatives the Young tableaux are very
useful in representation theory. They provide a convenient way to
describe the group representations of the symmetric and general
linear groups and to study their properties. In particular Young
diagrams are in one-to-one correspondence with the irreducible
representations of the symmetric group over the complex numbers and
the irreducible polynomial representations of the general linear
groups. They were introduced by Alfred Young in 1900. They were then
applied to the study of the symmetric group by Georg Frobenius in
1903. Their theory and applications were further developed  by many
mathematicians and there are  numerous and interesting applications,
beyond representation theory, in combinatorics and algebraic
geometry.

If we consider the subset all $\mm\in\hat\U(n+1)(\kk)$ such that
$m_{n+1}\ge0$ it is natural to represent such a state of our Markov
chain by its Young diagram, see Section \ref{Young}. Then in the
last two sections we describe a random mechanism based on Young
diagrams that gives rise to a random walk in the set of all Young
diagrams of $2n+1$ rows and whose $2j$ row  has $k_j$ boxes $1\le
j\le n$,  and whose transition matrix is $\tilde M_1$, see
\eqref{M1}.

\section{Spherical functions of $(\SU(n+1),\U(n))$}\label{esfericas}

Let $G$ be a locally compact unimodular group and let $K$ be a compact
subgroup of $G$. Let $\hat K$ denote the set of all
equivalence classes of complex finite dimensional irreducible
representations of $K$; for each $\delta\in \hat K$, let
$\xi_\delta$ denote the character of $\delta$, $d(\delta)$ the degree of
$\delta$, i.e. the dimension of any representation in
the  $\delta$, and $\chi_\delta=d(\delta)\xi_\delta$. We choose the Haar measure $dk$ on
$K$ normalized by $\int_K dk=1$.
We shall denote by $V$ a finite dimensional vector space over the field
$\CC$ of complex numbers and by $\End(V)$ the space
of all linear transformations of $V$ into $V$.

A {\em spherical function} $\Phi$ on $G$ of type $\delta\in \hat K$ is a
continuous function on $G$ with values in $\End(V)$ such
that
\begin{enumerate} \item[i)] $\Phi(e)=I$. ($I$= identity transformation).

\item[ii)] $\Phi(x)\Phi(y)=\int_K \chi_{\delta}(k^{-1})\Phi(xky)\, dk$, for
all $x,y\in G$.
\end{enumerate}

  If $\Phi:G\longrightarrow
\End(V)$ is a spherical function of type $\delta$
then $\Phi(kgk')=\Phi(k)\Phi(g)\Phi(k')$, for all $k,k'\in K$, $g\in G$,
and  $k\mapsto \Phi(k)$ is a representation of $K$ such that any
irreducible subrepresentation belongs to $\delta$.

Spherical functions of type $\delta$ arise in a natural way upon
considering representations of $G$. If $g\mapsto U(g)$ is a
continuous representation of $G$, say on a finite dimensional vector
space $E$, then $$P(\delta)=\int_K \chi_\delta(k^{-1})U(k)\, dk$$ is
a projection of $E$ onto $P(\delta)E=E(\delta)$. The function
$\Phi:G\longrightarrow \End(E(\delta))$ defined by
$$\Phi(g)a=P(\delta)U(g)a,\quad g\in G,\; a\in E(\delta)$$
is a spherical function of type $\delta$. In fact, if $a\in E(\delta)$ we have
\begin{align*}
\Phi(x)\Phi(y)a&= P(\delta)U(x)P(\delta)U(y)a=\int_K \chi_\delta(k^{-1})
P(\delta)U(x)U(k)U(y)a\, dk\\
&=\left(\int_K\chi_\delta(k^{-1})\Phi(xky)\, dk\right) a. \end{align*}
If the representation $g\mapsto U(g)$ is irreducible then the associated
spherical function $\Phi$ is also irreducible. Conversely, any irreducible
spherical function on a compact group $G$ arises in this way from a finite dimensional irreducible representation of $G$.

The aim of this section is to collect the necessary material to state and explain a three
term recursion relation for a sequence of matrix valued orthogonal polynomials, built up from
irreducible spherical functions of the same type associated to the pair
$(\SU(n+1), {\mathrm S}(\U(n)\times \U(1)))$.

The irreducible finite dimensional representations of $\SU(n+1) $ are restriction of  irreducible
representations of $\U(n+1)$, which are parameterized by $(n+1)$-tuples of integers
$$\mm=(m_{1}, m_{2},\dots, m_{n+1})$$
such that $m_{1}\geq m_{2}\geq \cdots \geq m_{n+1}$.

Different representations of $\U(n+1)$ can  restrict to the same
representation of $G=\SU(n+1)$. In fact the representations $\mm$
and $\pp$ of $\U(n+1)$ restrict to the same representation of
$\SU(n+1)$ if and only if $m_i=p_i+j$ for all $i=1,\dots,n+1$ and
some $j\in\ZZ$.

The closed subgroup $K={\mathrm S}(\U(n)\times \U(1))$ of $G$ is
isomorphic to $\U(n)$, hence its finite dimensional irreducible
representations are parameterized by the $n$-tuples of integers
$$\kk=(k_{1}, k_{2},\dots,k_{n})$$ subject to the conditions $k_{1}\geq k_{2}\geq \cdots \geq k_{n}$.

Let $\kk$ be an irreducible finite dimensional representation of
$\U(n)$. Then $\kk$ is a subrepresentation of $\mm$ if and only if
the coefficients $k_{i}$ satisfy the interlacing property
$$m_{i}\geq k_{i}\geq m_{i+1},\quad  \text { for all
}\quad  i=1, \dots , n.$$ Moreover if $\kk$ is a subrepresentation
of $\mm$ it appears only once. (See \cite{VK}).

The representation space $V_\kk$ of $\kk$ is a subspace of the
representation space $V_\mm$ of $\mm$ and it is also $K$-stable. In
fact, if $A\in\U(n)$, $a=(\det A)^{-1}$ and  $v\in V_\kk$ we have
$$\left(\begin{matrix} A&0\\0&a\end{matrix}\right)\cdot v=
a\left(\begin{matrix} a^{-1}A&0\\0&1\end{matrix}\right)\cdot
v=a^{s_\mm-s_\kk}\left(\begin{matrix}A&0\\0&1\end{matrix}\right)\cdot
v,$$ where $s_\mm= m_{1}+\cdots +m_{n+1}$ and $s_\kk=k_{1}+\cdots
+k_{n}$. This means that the representation of $K$ on $V_\kk$
obtained from $\mm$ by restriction is parameterized by
\begin{equation}\label{Ktipos}
(k_{1}+s_\kk-s_\mm, \dots ,k_{n} + s_\kk-
s_\mm).
\end{equation}

\

Let $\Phi^{\mm, \kk}$ be  the spherical function associated to the
representation $\mm$ of $G$ and to the subrepresentation $\kk$ of
$K$. Then \eqref{Ktipos} says that the $K$-type of $\Phi^{\mm, \kk}$ is
$\kk+(s_\kk-s_\mm)(1,\dots,1)$.

\begin{prop}\label{equivalencia}
  The spherical functions $\Phi^{\mm,\kk}$ and
  $\Phi^{\mm',\kk'}$ of the pair $(G,K)$ are equivalent if and only
  if $\mm'=\mm+j(1,\dots, 1)$ and $\kk'=\kk+j(1,\dots, 1)$.
\end{prop}
\begin{proof}
The spherical functions $\Phi^{\mm,\kk}$ and
  $\Phi^{\mm',\kk'}$ are equivalent if and only if
  $\mm$ and $\mm'$ are equivalent and the $K$-types of both spherical functions are the
  same, see the discussion in p. 85 of \cite{T1}.  We know that $\mm\simeq \mm'$ if and only if
$$\mm'=\mm+j(1,\dots, 1)\quad \text{ for some } j\in \ZZ.$$
Besides, the $K$ types are the same if and only if
$$k_{i}+s_\kk-s_\mm= k_{i}'+s_{\kk'}-s_{\mm'}\qquad \text{ for all } i=1,\dots,n.$$
Therefore $\kk'=\kk+p(1,\dots,1)$, and now it is easy to see that
$p=j$. \qed
\end{proof}

\

The standard representation of  $\U(n+1)$ on  $\CC^{n+1}$ is  irreducible and its
highest weight is
$(1,0,\dots , 0)$. Similarly the representation of  $\U(n+1)$ on the dual of
$\CC^{n+1}$ is  irreducible and its highest weight is $(0,\dots,0 , -1)$. Therefore we have that
$$\CC^{n+1}=V_{(1,0,\cdots, 0)}\quad \text{ and }\quad (\CC^n)^*=V_{(0,\dots,0,-1)}.$$
For any irreducible representation $\mm$ of  $\U(n+1)$  the tensor product
$V_\mm\otimes \CC^{n+1}$  decomposes as a direct sum of $\U(n+1)$-irreducible representations in the following way
\begin{equation}\label{tensor1}
  V_\mm\otimes \CC^{n+1}\simeq V_{\mm+\ee_1}\oplus V_{\mm+\ee_2}\oplus \cdots \oplus V_{\mm+\ee_{n+1}},
\end{equation}
and
\begin{equation}\label{tensor2}
  V_\mm\otimes (\CC^{n+1})^*\simeq V_{\mm-\ee_1}\oplus V_{\mm-\ee_2}\oplus \cdots \oplus V_{\mm-\ee_{n+1}},
\end{equation}
where $\{\ee_1,\cdots,\ee_{n+1}\}$ is the cannonical basis of $\CC^{n+1}$, see \cite{VK}.
\begin{remark} The irreducible modules on the right hand side of \eqref{tensor1} and \eqref{tensor2} whose
parameters $(m'_1,m'_2,\dots ,m'_{n+1})$ do not
satisfy the conditions $m_1'\geq m_2'\geq \dots \geq m_{n+1}'$ have to be omitted.
\end{remark}

Starting from \eqref{tensor1} and \eqref{tensor2}, the following theorem is proved in \cite{P}.

\begin{thm}\label{multiplicationformula}
 Let $\phi$ and $\psi$ be, respectively, the one dimensional spherical functions associated to the
 standard representation of $G$ and its dual. Then
\begin{align*}
\phi(g)\Phi^{\bf m,\kk}(g)&=\sum_{i=1}^{n+1}a_i^2(\mm, \kk) \Phi^{{\mathbf m}+{\mathbf
e}_i, \kk}(g)\\
\psi(g)\Phi^{\bf m, \kk}(g)&=\sum_{i=1}^{n+1}b_i^2(\mm, \kk) \Phi^{{\bf m}-{\bf
e}_i,\kk}(g).
\end{align*}
The constants $a_i(\mm, \kk)$ and $b_i(\mm, \kk)$  are given by
\begin{equation}\label{coefficientsab}
\begin{split}
a_i(\mm,\kk)&=\left\vert\frac{\prod_{j=1}^{n}(k_j-m_i-j+i-1)}
{\prod_{j\ne i}(m_j-m_i-j+i)}\right\vert^{1/2},\\
b_i(\mm,\kk)&=\left\vert\frac{\prod_{j=1}^{n}(k_j-m_i-j+i)}
{\prod_{j\ne i}(m_j-m_i-j+i)}\right\vert^{1/2}.
\end{split}
\end{equation}
Moreover
\begin{equation}\label{sumauno}
\sum_{i=1}^{n+1}a_i^2(\mm,\kk)=\sum_{i=1}^{n+1}b_i^2(\mm,\kk)=1.
\end{equation}
\end{thm}

\

Our Lie group  $G$ has the following polar decomposition $G=KAK$, where the abelian
subgroup A of G consists of all matrices of the form
\begin{equation}\label{atheta}
a= \left(\begin{matrix} \cos \theta& 0& \sin \theta\\ 0&I_{n-1}&0\\ -\sin \theta & 0&\cos
\theta\end{matrix}\right),\qquad \theta\in \RR.
\end{equation}
(Here $I_{n-1}$ denotes the identity matrix of size $n-1$). Since an
irreducible spherical function $\Phi$  of $G$ of type $\delta$
satisfies $\Phi(kgk')= \Phi(k)\Phi(g)\Phi(k')$ for all $k,k'\in K$
and $g\in G$, and $\Phi(k)$ is an irreducible representation of $K$
in the class $\delta$, it follows that $\Phi$ is determined by its
restriction to $A$ and its $K$-type. Hence, from now on, we shall
consider its restriction to $A$.

Let $M$
be the group consisting of all elements
of the form
$$m= \left(\begin{matrix} 1& 0&0\\ 0&B&0\\ 0&
0&1\end{matrix}\right), \qquad   B\in \U(n-1).$$
Thus $M$ is isomorphic to $\U(n-1)$ and its finite dimensional irreducible
representations are parameterized by the $(n-1)$-tuples of integers
$$\mathbf t=(t_{1}, t_{2},\dots, t_{n-1})$$
such that $t_{1}\geq t_{2}\geq \cdots \geq t_{n-1}$.

If $a\in A$, then  $\Phi^{\mm,\kk}(a)$ commutes with $\Phi^{\mm,\kk}(m)$
for all $m\in M$. In fact we have
$$\Phi^{\mm,\kk}(a)\Phi^{\mm,\kk}(m)=\Phi^{\mm,\kk}(am)=\Phi^{\mm,\kk}(ma)
=\Phi^{\mm,\kk}(m)\Phi^{\mm,\kk}(a).$$ The representation of $\U(n)$
in $V_\kk\subset V_\mm$, $\kk=(k_{1}, \dots ,k_{n})$ restricted to
$\U(n-1)$ decomposes as the following direct sum
\begin{equation}\label{Msubrepresentations}
V_\kk=\bigoplus_{\mathbf t\in \hat M} V_{\mathbf t},
\end{equation}
where the sum is over all the representations $\mathbf t=(t_{1},
\dots , t_{n-1})\in \hat M$ such that the coefficients of $\mathbf
t$ interlace the coefficients of $\kk$, that is $k_{i}\geq t_{i}\geq
k_{i+1}$, for all $i=1, \dots , n-1$. Since each $V_{\mathbf t}
\subset V_\kk$ appears only once, by Schur's Lemma, it follows that
$\Phi^{\mm,\kk}(a)|_{V_{\mathbf t}}=\phi^{\mm,\kk}_{\mathbf t}(a)
\text{Id}|_{V_{\mathbf t}}$, where $\phi^{\mm,\kk}_{\mathbf
t}(a)\in\CC$ for all $a\in A$.

\smallskip
By using Proposition \ref{equivalencia}, given a spherical function
$\Phi^{\mm,\kk}$ we can assume that
$s_\kk-s_\mm=0$. In such a case the $K$-type of $\Phi^{\mm,\kk}$ is
$\kk$, see \eqref{Ktipos}.
Now it is easy to see that if $(\mm,\kk)$  is one of such a pair then
\begin{equation}\label{n+1tupla}
\mm=\mm(w,\rr)=(w+k_{1},\, r_1+k_{2}, \, \dots \, , r_{n-1}+k_{n},
-(w+r_1+\cdots+ r_{n-1}) ),
\end{equation}
where  $0\leq w$, $k_{n}\geq -(w+r_1+\cdots+ r_{n-1})$ and $0\leq
r_i \leq k_{i}-k_{i+1}$ for $i=1,\dots n-1$. Thus if we assume
$w\ge\max\{0,-k_n\}$ and  $0\leq r_i \leq k_{i}-k_{i+1}$ for
$i=1,\dots n-1$ all the conditions are satisfied.

We observe that the representations $\mathbf t$ of $M$ appearing in
the right hand side of \eqref{Msubrepresentations} are of the form
$\mathbf t=\mathbf r+\kk'$, where $\kk'=(k_2,\dots,k_n)$ and
$\mathbf r$ is in the following set $$\Omega=\{\rr=(r_1, \dots,
r_{n-1}): 0\leq r_i\leq k_{i}-k_{i+1}\}.$$ In particular  the number
of $M$-modules in the decomposition of $V_\kk$ is
$$N=\prod_{i=1}^{n-1} (k_{i}-k_{i+1} +1).$$

 We will identify $\Phi^{\mm,\kk}(a)$ with the column vector
 $(\Phi^{\mm,\kk}_{\mathbf r}(a))_{\mathbf r\in\Omega}$ of $N$ complex valued functions
 $\Phi^{\mm,\kk}_{\mathbf r}(a)$ indexed by $\Omega$, where
 $\Phi^{\mm,\kk}_{\mathbf r}(a)=\phi^{\mm,\kk}_{\mathbf r+\kk'}(a)$, $a\in A$.

From now on we fix  $\kk\in\hat K$ and take $\mm=\mm(w,\rr)$ as in
\eqref{n+1tupla} for all $w\ge\max\{0,-k_n\}$ and $\rr\in\Omega$.
Also in the open subset $\{a(\theta)\in A:0<\theta<\pi/2\}$ of $A$,
we introduce the coordinate $t=\cos^2(\theta)$ and define on the
open interval $(0,1)$ the complex valued function $F_{\rr,\mathbf
s}(w,t)=\Phi^{\mm(w,\rr),\kk}_{\mathbf s}(a(\theta))$ and the
corresponding matrix function
$$F(w,t)=(F_{\rr,\mathbf s}(w,t))_{(\rr, \mathbf s)\in\Omega\times\Omega}.$$

For each $w\geq  \max\{0, -k_{n}\}$ we also define the following matrices of type $\Omega\times\Omega$
\begin{equation}\label{ABC}
A_w=((A_w)_{\rr,\mathbf s}), \quad B_w=((B_w)_{\rr,\mathbf s}), \quad C_w=((C_w)_{\rr,\mathbf s}),
\end{equation}
where
\begin{align*}
     (A_w)_{\rr,\mathbf s} & =\begin{cases}
       a^2_{n+1}(\mm(w,\rr))b^2_{1}(\mm(w,\rr)+\ee_{n+1}) & \quad\text{ if } \mathbf s=\rr\\
       a^2_{j+1}(\mm(w,\rr))b^2_{1}(\mm(w,\rr)+\ee_{j+1})  & \quad \text{ if } \mathbf s=\rr+\ee_j\\
       0 & \quad  \text{ otherwise}
     \end{cases} \displaybreak[0]
\\
     (C_w)_{\rr,\mathbf s} & =\begin{cases}
       a^2_{1}(\mm(w,\rr))b^2_{n+1}(\mm(w,\rr)+\ee_{1})) & \quad\text{ if } \mathbf s=\rr\\
       a^2_1(\mm(w,\rr))b^2_{j+1}(\mm(w,\rr)+\ee_1)  & \quad \text{ if } \mathbf s=\rr-\ee_j\\
       0 & \quad  \text{ otherwise}
     \end{cases} \displaybreak[0]
\\
     (B_w)_{\rr,\mathbf s} & =\begin{cases}
      \displaystyle \sum_{1\leq j\leq n+1}a^2_j(\mm(w,\rr))b^2_j(\mm(w,\rr)+\ee_j)) & \quad\text{ if } \mathbf s=\rr\\
       a^2_{j+1}(\mm(w,\rr))b^2_{n+1}(\mm(w,\rr)+\ee_{j+1})  & \quad \text{ if } \mathbf s=\rr+\ee_j\\
       a^2_{n+1}(\mm(w,\rr))b^2_{j+1}(\mm(w,\rr)+\ee_{n+1}) & \quad\text{ if } \mathbf s=\rr-\ee_j\\
       a^2_{j+1}(\mm(w,\rr))b^2_{i+1}(\mm(w,\rr)+\ee_{j+1}) & \quad\text{ if } \mathbf s=\rr+\ee_j-\ee_i \\
       0 & \quad  \text{ otherwise}
     \end{cases}
  \end{align*}
where $1\le i, j\le n-1$, and $a_i^2(\mm(w,\rr))=a_i^2(\mm,\kk)$,
$b_i^2(\mm(w,\rr)+\ee_j)=b_i^2((\mm+\ee_j,\kk))$ for $1\le i,j\le
n+1$, see \eqref {coefficientsab}.

\begin{thm} \label{multiplicationformula1} For each fixed $K$-type
 $\kk=(k_{1}, \dots k_{n})$, for all integers
 $w\ge\max\{0,-k_{n}\}$ and
all $0<t<1$ we have
\begin{equation}\label{3term}
tF(w,t)= A_wF(w-1,t)+ B_wF(w,t)+ C_wF(w+1,t).
\end{equation}
\end{thm}
\begin{proof}  This result is a consequence of Theorem \ref{multiplicationformula}
and of the appropriate definitions of $A_w, B_w, C_w$ given in
\eqref{ABC}, when we take $g=a(\theta)$.

 We recall that  $\phi(g)$ and $\psi(g)$ are the one dimensional spherical
 functions associated to the $G$-modules $\CC^{n+1}$ and $(\CC^{n+1})^*$, respectively.
 A direct computation gives
$$\phi(a(\theta))=\langle a(\theta)e_{n+1}, e_{n+1}\rangle =\cos \theta. $$
and $$\psi(a(\theta))=\langle a(\theta)\lambda_{n+1},
\lambda_{n+1}\rangle =\cos \theta.$$ Then
$\phi(a(\theta))\psi(a(\theta))=\cos^2(\theta)=t$. \qed
\end{proof}

\

If $g\in G=\SU(n+1)$ let $A(g)$ denote the $n\times n$ left upper
corner of  $g$, and let $\A$ be the dense open subset of all $g\in
G$ such that  $A(g)$ is nonsingular. As in \cite{PT3} in order to
determine all irreducible spherical functions of $G$ of type
$\kk=(k_1,\dots, k_n)$ an auxiliary  function
$\Phi_\kk:\A\longrightarrow \End(V_\kk)$ is introduced. It is
defined by $\Phi_\kk(g)=\pi(A(g))$ where  $\pi$ stands for the
unique holomorphic representation of $\GL(n,\CC)$ corresponding to
the parameter $\kk$.  It turns out that if $k_{n}\geq 0$ then
$\Phi_\kk =\Phi^{\mm,\kk}$ where $\mm=(k_1,\dots,k_n,0)$.

Then instead of looking at a general spherical function
$\Phi^{w,\rr}=\Phi^{\mm(w,\rr),\kk}$ of type $\kk$ we look at the
function $H^{w,\rr}(g)=\Phi^{w,\rr}(g)\Phi_\kk (g)^{-1}$ which is
well defined on $\A$.

As before we construct the matrix function
$$\tilde H(w,t)=(\tilde H_{\mathbf r,\mathbf s}(w,t))_{(\mathbf r,\mathbf s)\in\Omega\times\Omega}.$$
where $\tilde H_{\mathbf r,\mathbf s}(w,t)=H^{w,\mathbf r}_{\mathbf s}(a(\theta))$, $t=\cos\theta\in(0,1)$.

\

Let $\Psi(t)=(\Psi_{\mathbf r,\mathbf s}(t))_{(\mathbf r,\mathbf
s)\in\Omega\times\Omega}$ be the transpose of $\tilde H(0,t)$, i.e.
$\Psi_{\mathbf r,\mathbf s}(t)=\tilde H_{\mathbf s,\mathbf r}(0,t)$.
In \cite{PT3} the following crucial theorem is proved.

\begin{thm} If $k_n\ge0$, then $\tilde H_{\mathbf r,\mathbf s}(w,t)$, $\tilde H(w,t)$ and
$$\tilde P_w(t)=\tilde H(w,t)\Psi(t)^{-1}$$
are  polynomial functions on the variable $t$ whose degrees are
\begin{equation}\label{degree}
\begin{split}
\deg\tilde H_{\mathbf r,\mathbf s}(w,t)&=w+\sum_{i=1}^{n-1}\min\{r_i,s_i\},\\
\deg\tilde H(w,t)&=w+k_1-k_n,\\
\deg\tilde P_w(t)&=w.
\end{split}
\end{equation}
\end{thm}
It is important to point out that  $\{\tilde P_w\}_{w\ge0}$ is a
sequence of matrix orthogonal polynomials  with respect to a matrix
weight function $W=W(t)$ supported in the interval $(0,1)$ and given
in \cite{PT3}. From \eqref{3term} it easily follows that $\{\tilde
P_w\}_{w\ge0}$ satisfies the following three term recursion relation
\begin{equation}\label{4term}
t \tilde P_w(t)= A_w\tilde P_{w-1}(t)+ B_w\tilde P_w(t)+ C_w\tilde P_{w+1}(t).
\end{equation}
The above three term recursion relation  which hold for all $w\ge0$
can be written in the following way
\begin{equation}\label{matrix}
t\begin{vmatrix}\tilde P_0\\\tilde P_1\\\tilde P_2\\\tilde P_3\\\cdot\end{vmatrix}
=\begin{vmatrix} B_0&C_0&0&\\A_1&B_1&C_1&0&\\0&A_2&B_2&C_2&0&\\ &0&A_3&B_3&C_3&0\\&&\cdot&\cdot&\cdot&\cdot&\cdot\end{vmatrix}
\begin{vmatrix}\tilde P_0\\\tilde P_1\\\tilde P_2\\\tilde P_3\\\cdot\end{vmatrix}.
\end{equation}

Now we observe that the semi-infinite matrix $M$ on the right hand
side is a stochastic matrix, i.e. all the entries are nonnegative
and the sum of the elements in any row is one. In fact, the elements
in the $\rr$ row of the $w$ blocks are either zero or
$(A_w)_{\rr,\mathbf s}$, $(B_w)_{\rr,\mathbf s}$,
$(C_w)_{\rr,\mathbf s}$ which are given in \eqref{ABC}. Their sum is
\begin{align*}
\sum_{\mathbf s\in\Omega}(A_w)_{\rr,\mathbf s}&+(B_w)_{\rr,\mathbf
s}+(C_w)_{\rr,\mathbf s}=a^2_{n+1}(\mm)b^2_{1}(\mm+\ee_{n+1})\displaybreak[0]\\
&+\sum_{j=2}^n a^2_{j}(\mm)b^2_{1}(\mm+\ee_{j})
+\sum_{j=1}^{n+1}a^2_j(\mm)b^2_j(\mm+\ee_j)\displaybreak[0]\\
&+\sum_{j=2}^n
       a^2_{j}(\mm)b^2_{n+1}(\mm+\ee_{j})
       +a^2_{n+1}(\mm)\sum_{j=2}^n
       b^2_{j}(\mm+\ee_{n+1})\displaybreak[0]\\
       &+\sum_{2\le i\ne j\le
       n}a^2_{j}(\mm)b^2_{i}(\mm+\ee_{j})
       +a^2_{1}(\mm)b^2_{n+1}(\mm+\ee_{1})\displaybreak[0]\\
       &+a^2_1(\mm)\sum_{j=2}^n b^2_{j}(\mm+\ee_1),
       \end{align*}
where we replaced $\mm(w,\rr)$ by $\mm$. The right hand side can be
rewritten to obtain
\begin{align*}
\sum_{\mathbf s\in\Omega}(A_w)_{\rr,\mathbf s}&+(B_w)_{\rr,\mathbf
s}+(C_w)_{\rr,\mathbf
s}=a^2_{n+1}(\mm)\sum_{j=1}^{n+1}b^2_{j}(\mm+\ee_{n+1})\\
&+\sum_{j=2}^n
a^2_{j}(\mm)\sum_{i=1}^{n+1}b^2_{i}(\mm+\ee_{j})
+a^2_{1}(\mm)\sum_{j=1}^{n+1}b^2_{n+1}(\mm+\ee_{1})\displaybreak[0]\\
&\hskip3.3cm=\sum_{j=1}^{n+1}a^2_{j}(\mm)\sum_{i=1}^{n+1}b^2_{i}(\mm+\ee_{j}).
\end{align*}
Now by using \eqref{sumauno} the assertion
$$\sum_{\mathbf s\in\Omega}(A_w)_{\rr,\mathbf s}+(B_w)_{\rr,\mathbf s}+(C_w)_{\rr,\mathbf s}=1$$
follows, proving that the semi-infinite matrix $M$ is stochastic.

\section{The substeps of the random walk}\label{factorizacion}

In what follows we will construct a factorization of the stochastic
matrix $M$ appearing in \eqref{matrix} into the product of two
stochastic matrices of the form

\begin{equation}\label{factorization}
M=\begin{vmatrix} Y_0&X_0&0&\\0&Y_1&X_1&0&\\&0&Y_2&X_2&0&\\
&&0&Y_3&X_3&0\\&&&\cdot&\cdot&\cdot&\cdot\end{vmatrix}
\begin{vmatrix} S_0&0&&\\R_1&S_1&0&&\\0&R_2&S_2&0&&\\
&0&R_3&S_3&0&\\&&\cdot&\cdot&\cdot&\cdot\end{vmatrix}.
\end{equation}

While the random process given by the matrix $M$ leaves invariant the set $P$ introduced below, see (28),
this is not true for its substeps going along with this factorization. This section deals with this
complication in great detail.

The multiplication formulas given in Theorem
\ref{multiplicationformula} restricted to $g=a(\theta)$ give
\begin{equation}\label{multi0}
\begin{split}
\cos(\theta)\Phi^{\mm,\kk}_{\mathbf
s}(a(\theta))&=\sum_{j=1}^{n+1}a^2_j(\mm,\kk)\Phi^{\mm+\ee_j,\kk}_{\mathbf
s}(a(\theta)),\\
\cos(\theta)\Phi^{\mm,\kk}_{\mathbf
s}(a(\theta))&=\sum_{j=1}^{n+1}b^2_j(\mm,\kk)\Phi^{\mm-\ee_j,\kk}_{\mathbf
s}(a(\theta)).
\end{split}
\end{equation}
We recall that we fixed $\kk$ with $k_n\ge0$ and we took $\mm=\mm(w,\rr)$ as in
\eqref{n+1tupla}. Also making the change of variables
$t=\cos(\theta)$ we defined $F_{\rr,\mathbf
s}(w,t)=\Phi^{\mm(w,\rr),\kk}_{\mathbf s}(a(\theta))$. Now we make
the following important observation
\begin{equation}\label{importante}
\begin{split}
     \mm(w,\rr)\pm\ee_j& =\begin{cases}
       \mm(w\pm1,\rr)\pm\ee_{n+1} & \quad\text{ if } j=1,\\
       \mm(w,\rr\pm\ee_{j-1})\pm\ee_{n+1}  & \quad \text{ if } j=2,\dots,n,\\
       \mm(w,\rr)\pm\ee_{n+1} & \quad\text{ if } j=n+1.
     \end{cases}
  \end{split}
  \end{equation}
Introduce the following scalar functions
$$F^+_{\rr,\mathbf
s}(w,t)=\Phi^{\mm(w,\rr)+\ee_{n+1},\kk}_{\mathbf s}(a(\theta)),$$
and the matrix function
$$F^+(w,t)=(F^+_{\rr,\mathbf s}(w,t))_{(\rr,\mathbf
s)\in\Omega\times\Omega}.$$ Then the first identity in
\eqref{multi0} becomes
\begin{equation}\label{multi1}
\begin{split}
\sqrt tF_{\rr,\mathbf s}(w,t)&=a_1^2(\mm(w,\rr))F^+_{\rr,\mathbf s}(w+1,t)\\
+\sum_{j=1}^{n-1}&a^2_{j+1}(\mm(w,\rr))F^+_{\rr+\ee_j,\mathbf
s}(w,t)+a_{n+1}^2(\mm(w,\rr))F^+_{\rr,\mathbf s}(w,t).
\end{split}
\end{equation}

For each $w\ge0$ we define the following matrix
of type $\Omega\times\Omega$
\begin{equation}\label{XY}
X_w=((X_w)_{\rr,\mathbf s}), \quad Y_w=((Y_w)_{\rr,\mathbf s}),
\end{equation}
where
\begin{align*}
     (X_w)_{\rr,\mathbf s} & =\begin{cases}
       a^2_{1}(\mm(w,\rr))& \quad\text{ if } \mathbf s=\rr,\\
       0 & \quad  \text{ otherwise},
     \end{cases} \displaybreak[0]
\\
     (Y_w)_{\rr,\mathbf s} & =\begin{cases}
       a^2_{n+1}(\mm(w,\rr)) & \quad\text{ if } \mathbf s=\rr,\\
       a^2_{j+1}(\mm(w,\rr)) & \quad \text{ if } \mathbf s=\rr+\ee_j,\\
       0 & \quad  \text{ otherwise}.
     \end{cases} \displaybreak[0]
  \end{align*}

Now the set of scalar identities \eqref{multi1} with $(\rr,{\bf
s})\in \Omega\times\Omega$ can be written as a matrix identity in
the following more convenient way
\begin{equation}\label{multi2}
\sqrt tF(w,t)=X_wF^+(w+1,t)+Y_wF^+(w,t).
\end{equation}

For each $w\ge0$ we define the following matrix
of type $\Omega\times\Omega$
\begin{equation}\label{RS}
R_w=((R_w)_{\rr,\mathbf s}), \quad S_w=((S_w)_{\rr,\mathbf s}),
\end{equation}
where
\begin{align*}
     (R_w)_{\rr,\mathbf s} & =\begin{cases}
       b^2_{1}(\mm(w,\rr)+\mathbf e_{n+1})& \quad\text{ if } \mathbf s=\rr,\\
       0 & \quad  \text{ otherwise},
     \end{cases} \displaybreak[0]
\\
     (S_w)_{\rr,\mathbf s} & =\begin{cases}
       b^2_{n+1}(\mm(w,\rr)+\mathbf e_{n+1}) & \quad\text{ if } \mathbf s=\rr,\\
       b^2_{j+1}(\mm(w,\rr)+\mathbf e_{n+1}) & \quad \text{ if } \mathbf s=\rr-\ee_j,\\
       0 & \quad  \text{ otherwise}.
     \end{cases} \displaybreak[0]
  \end{align*}

\noindent If we multiply \eqref{multi2} by $\sqrt t$ and use the
second multiplication formula given in \eqref{multi0} we obtain
\begin{equation}\label{multi5}
\begin{split}
tF(w,t)=&X_w(R_{w+1}F(w,t)+S_{w+1}F(w+1,t))\\
&+Y_w(R_wF(w-1,t)+S_wF(w,t))\\
=&(X_wR_{w+1}+Y_wS_w)F(w,t)+X_wS_{w+1}F(w+1,t)\\
&+Y_wR_wF(w-1,t),
\end{split}
\end{equation}
since we claim that
\begin{equation}\label{multiplicationformula10}
\sqrt tF^+(w,t)=R_wF(w-1,t)+S_wF(w,t).
\end{equation}
Indeed we have
\begin{equation*}
\begin{split}
\sqrt tF^+_{\mathbf r,\mathbf s}(w,t)=&\sqrt
t\Phi^{\mm(w,\rr)+\mathbf
e_{n+1},\kk}_{\mathbf s}(a(\theta))\\
=&\sum_{j=1}^{n+1}b_j^2(\mm(w,\rr)+\mathbf e_{n+1})
\Phi^{\mm(w,\rr)+\mathbf e_{n+1}-\mathbf e_j,\kk}_{\mathbf
s}(a(\theta))\\
=&b_1^2(\mm(w,\rr)+\mathbf e_{n+1}) \Phi^{\mm(w-1,\rr),\kk}_{\mathbf
s}(a(\theta))\\
&+\sum_{j=2}^{n}b_j^2(\mm(w,\rr)+\mathbf e_{n+1})
\Phi^{\mm(w,\rr-\mathbf e_{j-1},\kk}_{\mathbf s}(a(\theta))\\
&+b_{n+1}^2(\mm(w,\rr)+\mathbf e_{n+1})
\Phi^{\mm(w,\rr),\kk}_{\mathbf s}(a(\theta)),
\end{split}
\end{equation*}
where we used \eqref{importante}.

On the other hand
\begin{equation*}
\begin{split}
(R_wF(w-1,t))_{\mathbf r,\mathbf s}=&\sum_{q\in\Omega}(R_w)_{\mathbf
r,\mathbf q}F_{\mathbf q,\mathbf s}(w-1,t)\\
=&b_1^2(\mm(w,\rr)+\mathbf e_{n+1})F_{\mathbf r,\mathbf s}(w-1,t),
\end{split}
\end{equation*}
and
\begin{equation*}
\begin{split}
(S_wF(w,t))_{\mathbf r,\mathbf s}=&\sum_{q\in\Omega}(S_w)_{\mathbf
r,\mathbf q}F_{\mathbf q,\mathbf s}(w,t)\\
=&b_{n+1}^2(\mm(w,\rr)+\mathbf e_{n+1})F_{\mathbf r,\mathbf
s}(w,t)\\
&+\sum_{j=1}^{n-1}b_{j+1}^2(\mm(w,\rr)+\mathbf e_{n+1})F_{\mathbf
r-\mathbf e_j,\mathbf s}(w,t).
\end{split}
\end{equation*}
Then \eqref{multiplicationformula10} follows easily.

Finally if we compare \eqref{multi5} with \eqref{3term} in Theorem
\ref{multiplicationformula1} we obtain
$$A_w=Y_wR_w,\quad B_w=X_wR_{w+1}+Y_wS_w,\quad C_w=X_wS_{w+1}$$
which is equivalent to the factorization \eqref{factorization}.

We end by checking that both matrices in the right hand side of
\eqref{factorization} are stochastic:
\begin{align*}
\sum_{\mathbf s\in\Omega}& (Y_w)_{\rr,\mathbf s}+\sum_{\mathbf
s\in\Omega}(X_w)_{\rr,\mathbf s}\\ &\qquad =a_{n+1}^2(\mm(w,\rr))
+\sum_{1\leq j\leq n-1}a_{j+1}^2(\mm(w,\rr))+a_1^2(\mm(w,\rr))=1,\displaybreak[0]\\
\sum_{\mathbf s\in\Omega}&(R_w)_{\rr,\mathbf s}+\sum_{\mathbf
s\in\Omega}(S_w)_{\rr,\mathbf s}\\ &\qquad
=b_{1}^2(\mm(w,\rr)+\mathbf e_{n+1}) +b_{n+1}^2(\mm(w,\rr)+\mathbf
e_{n+1})\\
&\hskip1.2cm+\sum_{1\leq j\leq n-1}b_{j+1}^2(\mm(w,\rr)+\mathbf
e_{n+1})=1,
\end{align*}
where we used that
$\sum_{i=1}^{n+1}a_i^2(\mm,\kk)=\sum_{i=1}^{n+1}b_i^2(\mm,\kk)=1$,
see \eqref{sumauno}.

\

Now we want to consider the random walks associated to the
probability matrices appearing in \eqref{factorization},
$$M=\begin{vmatrix} B_0&C_0&0&\\A_1&B_1&C_1&0&\\0&A_2&B_2&C_2&0&\\
&\cdot&\cdot&\cdot&\cdot&\cdot&\end{vmatrix}=M_1M_2,$$
\begin{equation}\label{M1}
M_1=\begin{vmatrix} Y_0&X_0&0&\\0&Y_1&X_1&0&\\&0&Y_2&X_2&0&\\
&&\cdot&\cdot&\cdot&\cdot\end{vmatrix},
\quad M_2=\begin{vmatrix} S_0&0&&\\R_1&S_1&0&&\\0&R_2&S_2&0&&\\
&\cdot&\cdot&\cdot&\cdot\end{vmatrix}.
\end{equation}

Let $F_w$ and $F^+_w$ denote, respectively, the polynomial functions
$F_w=F_w(t)$ and $F^+_w=F^+_w(t)$. Then
\eqref{multiplicationformula10} can be written as follows
\begin{equation}\label{W2}
\sqrt t\begin{vmatrix} F^+_0\\ F^+_1\\ F^+_2\\\cdot\end{vmatrix}
=\begin{vmatrix} S_0&0&\\R_1&S_1&0&\\0&R_2&S_2&0&\\
&\cdot&\cdot&\cdot&\cdot&\end{vmatrix}
\begin{vmatrix} F_0\\ F_1\\ F_2\\\cdot\end{vmatrix}.
\end{equation}
Similarly \eqref{multi2} gives
\begin{equation}\label{W1}
\sqrt t\begin{vmatrix} F_0\\ F_1\\ F_2\\\cdot\end{vmatrix}
=\begin{vmatrix} Y_0&X_0&0&\\0&Y_1&X_1&0&\\&0&Y_2&X_2&0&\\
&&\cdot&\cdot&\cdot&\cdot&\end{vmatrix}
\begin{vmatrix} F^+_0\\ F^+_1\\ F^+_2\\\cdot\end{vmatrix}.
\end{equation}
We can now rewrite \eqref{multi5} in matrix form,
\begin{equation}\label{W3}
t\begin{vmatrix} F_0\\ F_1\\ F_2\\\cdot\end{vmatrix}=
\sqrt t M_1\begin{vmatrix} F^+_0\\ F^+_1\\ F^+_2\\
\cdot\end{vmatrix}=M_1M_2\begin{vmatrix} F_0\\ F_1\\ F_2\\
\cdot\end{vmatrix}=M\begin{vmatrix} F_0\\ F_1\\ F_2\\\cdot\end{vmatrix}.
\end{equation}

The state space of the random walks $W,W_1,W_2$ associated,
respectively, to the stochastic matrices $M,M_1,M_2$ is the set
$\NN_{\ge0}\times\Omega$, and $W$ is equal to the composition
$W_1\circ W_2$.

We recall that the map $(w,\rr)\mapsto \mm(w,\rr)$ defined in
\eqref{n+1tupla} is an injection of $\NN_{\ge0}\times\Omega$ into
the $\kk$-spherical dual $\hat\U(n+1)(\kk)$ of $\U(n+1)$, and its
image is
\begin{equation}\label{P}
P=\{\mm\in\hat\U(n+1)(\kk):s_\mm=s_\kk\},
\end{equation}
where $s_\mm=m_1+\cdots+m_{n+1}, s_\kk=k_1+\cdots+ k_n$.

Let us now consider the random walk $W_1$ associated to the
stochastic matrix $M_1$. Below we display the entries of $M_1$ at
the different sites of its $(w,\mathbf r)$-row,
\begin{equation*}
 \begin{cases}
       a^2_{n+1}(\mm(w,\rr)) & \quad\text{ if  $\mm(w,\mathbf s)$-site}=\mm(w,\rr),\\
       a^2_{j+1}(\mm(w,\rr)) & \quad \text{ if  $\mm(w, \mathbf s)$-site}=\mm(w,\rr+\ee_j),\\
       a^2_1(\mm(w,\rr)) & \quad \text{ if  $\mm(w, \mathbf s)$-site}=\mm(w+1,\rr),\\
       0 & \quad  \text{ in  other sites}.
     \end{cases} \displaybreak[0]
\end{equation*}
The appearance of the plus sign in the right hand side of \eqref{W1}
makes it natural to consider instead the random walk $W_1^+$ obtained from
$W_1$ by applying a shift by $\ee_{n+1}$. Thus, if the system is at
state $\mm(w,r)$ at time $t$, then at time $t+1$ it can move in the
following ways
\begin{equation*}
W_1^+: \begin{cases}
       \mm(w,\rr) \rightarrow  \mm(w,\rr) + \ee_{n+1},&\text{ with probability } \; a^2_{n+1}(\mm(w,\rr)),\\
       \mm(w,\rr) \rightarrow \mm(w,\rr) + \ee_{j+1},&\text{ with probability } \; a^2_{j+1}(\mm(w,\rr)),\\
       \mm(w,\rr) \rightarrow \mm(w, \rr) +\ee_1,& \text{ with probability } \; a^2_1(\mm(w,\rr)),\\
       \mm(w,\rr) \rightarrow  \text{other states},&\text{ with probability } \; 0,
       \end{cases} \displaybreak[0]
\end{equation*}
because $\mm(w,\rr+\ee_j) +\ee_{n+1}= \mm(w,\rr) + \ee_{j+1}$ for
$1\le j\le n-1$, and $\mm(w+1,\rr) +\ee_{n+1}=\mm(w,\rr) +\ee_1$.
This is in accordance with the following formula derived by looking
at the $((w,\rr),\mathbf s)$-entry of  \eqref{W1},
$$\cos(\theta) \Phi_{\mathbf s}^{\mm(w,\rr),\kk}(a(\theta))=\sum_{j=1}^{n+1}a^2_j(\mm(w,\rr))\Phi_{\mathbf s}^{\mm(w,\rr)+\ee_j,\kk}(a(\theta)).$$

Now it is worth to observe that $W_1^+$ does not leave invariant the
subset $P$ but extends to a random walk $ \tilde W_1$ in
$\hat\U(n+1)(\kk)$ defined by
\begin{equation}\label{W1plus}
\tilde W_1: \mm\rightarrow \mm+\ee_j, \text{ with probability } \;
a^2_j(\mm,\kk).
\end{equation}

We proceed similarly with the random walk $W_2$ associated to the
stochastic matrix $M_2$. Below we display the entries of $M_2$ at
the different sites of its $(w,\mathbf r)$-row,
\begin{equation*}
 \begin{cases}
       b^2_{n+1}(\mm(w,\rr)+\ee_{n+1}) & \quad\text{ if  $\mm(w,\mathbf s)$-site}=\mm(w,\rr),\\
       b^2_{j+1}(\mm(w,\rr)+\ee_{n+1}) & \quad \text{ if  $\mm(w, \mathbf s)$-site}=\mm(w,\rr-\ee_j),\\
       b^2_1(\mm(w,\rr)+\ee_{n+1}) & \quad \text{ if  $\mm(w, \mathbf s)$-site}=\mm(w-1,\rr),\\
       0 & \quad  \text{ in  other sites}.
     \end{cases} \displaybreak[0]
\end{equation*}
The appearance of the plus sign in the left hand side of \eqref{W2}
makes it natural to consider instead the random walk $W_2^-$ obtained from
$W_2$ by applying a shift by $-\ee_{n+1}$. Thus, if the system is at
state $\mm(w,r)$ at time $t$, then at time $t+1$ it can move in the
following ways
\begin{equation*}
W_2^-: \begin{cases}
       \mm(w,\rr) \rightarrow  \mm(w,\rr) - \ee_{n+1},&\text{ with prob. } \; b^2_{n+1}(\mm(w,\rr)+\ee_{n+1}),\\
       \mm(w,\rr) \rightarrow \mm(w,\rr) - \ee_{j+1},&\text{ with prob. } \; b^2_{j+1}(\mm(w,\rr)+\ee_{n+1}),\\
       \mm(w,\rr) \rightarrow \mm(w, \rr) - \ee_1,& \text{ with prob. } \; b^2_1(\mm(w,\rr)+\ee_{n+1}),\\
       \mm(w,\rr) \rightarrow  \text{other states},&\text{ with prob. } \; 0,
       \end{cases} \displaybreak[0]
\end{equation*}
because $\mm(w,\rr-\ee_j) - \ee_{n+1}= \mm(w,\rr) - \ee_{j+1}$ for
$1\le j\le n-1$, and $\mm(w-1,\rr) - \ee_{n+1}=\mm(w,\rr) - \ee_1$.
This is in accordance with the following formula derived by looking
at the $((w,\rr),\mathbf s)$-entry of  \eqref{W2},
$$\cos(\theta) \Phi_{\mathbf s}^{\mm(w,\rr)+\ee_{n+1},\kk}(a(\theta))=\sum_{j=1}^{n+1}b^2_j(\mm(w,\rr)+\ee_{n+1})
\Phi_{\mathbf s}^{\mm(w,\rr)+\ee_{n+1}-\ee_j,\kk}(a(\theta)).$$

Then  $W_2^-$ does not leave invariant the subset $P$ but extends to
a random walk $\tilde W_2$ in $\hat\U(n+1)(\kk)$ defined by
\begin{equation}\label{W2minus}
\tilde W_2: \mm\rightarrow \mm-\ee_j, \text{ with probability } \;
b^2_j(\mm+\ee_{n+1},\kk),
\end{equation}
for $1\le j\le n+1$.

The transition matrices of $\tilde W_1$ and $\tilde W_2$ are, respectively,
the following block bidiagonal matrices
\begin{equation}\label{tildeM1}
\tilde M_1=\begin{vmatrix} \tilde Y_0&\tilde X_0&0&\\0&\tilde Y_1&\tilde X_1&0&\\&0&\tilde Y_2&\tilde X_2&0&\\
&&\cdot&\cdot&\cdot&\cdot\end{vmatrix},
\quad \tilde M_2=\begin{vmatrix} \tilde S_0&0&&\\\tilde R_1&\tilde S_1&0&&\\0&\tilde R_2&\tilde S_2&0&&\\
&\cdot&\cdot&\cdot&\cdot\end{vmatrix},
\end{equation}
with
\begin{align*}
     (\tilde X_w)_{\mm,\nn} & =\begin{cases}
       a^2_{1}(\mm)& \quad\text{ if } \nn=\mm,\\
       0 & \quad  \text{ otherwise},
     \end{cases} \displaybreak[0]
\\
     (\tilde Y_w)_{\mm,\nn} & =\begin{cases}
       a^2_{n+1}(\mm) & \quad\text{ if } \nn=\mm,\\
       a^2_{j+1}(\mm) & \quad \text{ if } \nn=\mm+\ee_j,\\
       0 & \quad  \text{ otherwise},
     \end{cases} \displaybreak[0]
  \end{align*}
\begin{align*}
     (\tilde R_w)_{\mm,\nn} & =\begin{cases}
       b^2_{1}(\mm+\mathbf e_{n+1})& \quad\text{ if } \nn=\mm,\\
       0 & \quad  \text{ otherwise},
     \end{cases} \displaybreak[0]
\\
     (\tilde S_w)_{\mm,\nn} & =\begin{cases}
       b^2_{n+1}(\mm+\mathbf e_{n+1}) & \quad\text{ if } \nn=\mm,\\
       b^2_{j+1}(\mm+\mathbf e_{n+1}) & \quad \text{ if } \nn=\rr-\ee_j,\\
       0 & \quad  \text{ otherwise}.
     \end{cases} \displaybreak[0]
  \end{align*}
where $\mm,\nn\in\hat\U(n+1)(\kk)$ are such that
$w=m_1-k_1=n_1-k_1$, and $1\le j\le n-1$.

Moreover, the stochastic matrix $\tilde M$ corresponding to the
composition $\tilde W=\tilde W_1\circ\tilde W_2$ is equal to $\tilde
M_1\tilde M_2$, and it is given by
$$\tilde M=\begin{vmatrix} \tilde B_0&\tilde C_0&0&\\\tilde A_1&\tilde B_1&\tilde C_1&0&\\0&\tilde A_2&\tilde B_2&\tilde C_2&0&\\
&\cdot&\cdot&\cdot&\cdot&\cdot&\end{vmatrix},$$

with
\begin{align*}
     (\tilde A_w)_{\mm,\nn} & =\begin{cases}
       a^2_{n+1}(\mm)b^2_{1}(\mm+\ee_{n+1}) & \quad\text{ if } \nn=\mm\\
       a^2_{j+1}(\mm)b^2_{1}(\mm+\ee_{j+1})  & \quad \text{ if } \nn=\mm+\ee_j\\
       0 & \quad  \text{ otherwise}
     \end{cases} \displaybreak[0]
\\
     (\tilde C_w)_{\mm,\nn} & =\begin{cases}
       a^2_{1}(\mm)b^2_{n+1}(\mm+\ee_{1})) & \quad\text{ if } \nn=\mm\\
       a^2_1(\mm)b^2_{j+1}(\mm+\ee_1)  & \quad \text{ if } \nn=\mm-\ee_j\\
       0 & \quad  \text{ otherwise}
     \end{cases} \displaybreak[0]
\\
     (\tilde B_w)_{\mm,\nn} & =\begin{cases}
      \displaystyle \sum_{1\leq j\leq n+1}a^2_j(\mm)b^2_j(\mm+\ee_j)) & \quad\text{ if } \nn=\mm\\
       a^2_{j+1}(\mm)b^2_{n+1}(\mm+\ee_{j+1})  & \quad \text{ if } \nn=\mm+\ee_j\\
       a^2_{n+1}(\mm)b^2_{j+1}(\mm+\ee_{n+1}) & \quad\text{ if } \nn=\mm-\ee_j\\
       a^2_{j+1}(\mm)b^2_{i+1}(\mm+\ee_{j+1}) & \quad\text{ if } \nn=\mm+\ee_j-\ee_i \\
       0 & \quad  \text{ otherwise}
     \end{cases}
  \end{align*}
where $\mm,\nn\in\hat\U(n+1)(\kk)$ are such that
$w=m_1-k_1=n_1-k_1$, and $1\le i, j\le n-1$. The coefficients $a^2_i(\mm), b^2_i(\mm)$
for $1\le i\le n+1$ are those defined in  \eqref {coefficientsab}.

  If we identify $\NN_{\ge0}\times\Omega$ with the subset $P$, defined in \eqref{P},
  by $(w,\rr)\equiv \mm(w,\rr)$, then clearly
  $W=\tilde W_{|P}$, because $M$ become a submatrix of $\tilde M$. Therefore
  $$W_1\circ W_2=W=\tilde W_{|P}=(\tilde W_1\circ\tilde W_2)_{|P}.$$

To conclude, the analysis of the random walk $W$ associated to the stochastic matrix $M$
is simplified  by looking at the decomposition $W= (\tilde W_1\circ\tilde W_2)_{|P}$
instead of considering $W=W_1\circ W_2$.

\section{An urn model for $\U(3)$}\label{urnmodelU3}

We now give a concrete probabilistic mechanism that goes along with
the random walk $\tilde W_1$ constructed in Section
\ref{factorizacion} by group theoretical means, see \eqref{W1plus}.
An entirely similar construction going with $\tilde W_2$ can be
considered for the other substep of our process.

This section is included for the benefit of the reader. It describes
in detail, for the simple case of $n=2$ going along with the pair
$(\U(3),\U(2))$, a construction that will be given in general in
Section \ref {urnmodelUn}.

A configuration, or state of our system, is now a triple of integers
$\mm=(m_1,m_2,m_3)$ subject to the constrains $m_1\ge k_1\ge m_2\ge
k_2\ge  m_3$ with two fixed integers $k_1\ge k_2$, see
\eqref{constrains}. We describe a stochastic mechanism whereby  one
of the three values of the $m_i$ is incresased by one with the
following probabilities, see \eqref{coefficientsab}
\begin{equation*}
\begin{split}
a_1^2(\mm,\kk)&=\frac{(m_1-k_1+1)(m_1-k_2+2)}{(m_1-m_2+1)(m_1-m_3+2)},\\
a_2^2(\mm,\kk)&=\frac{(k_1-m_2)(m_2-k_2+1)}{(m_1-m_2+1)(m_2-m_3+1)},\\
a_3^2(\mm,\kk)&=\frac{(k_1-m_3+1)(k_2-m_3)}{(m_1-m_3+2)(m_2-m_3+1)}.
\end{split}
\end{equation*}

In the general scheme to be considered later
this case corresponds to the value $n=2$, and thus we start with
two urns $B_1,B_2$. In urn $B_j$, $j=1,2$, place $m_j-k_j+1$ balls of color
$c_j$ and $k_j-m_{j+1}$ balls of color $d_j$. These four colors are all
different. Notice that we could have no balls of colors $d_1$ or $d_2$
and that the total number of balls in urn $B_j$ is $m_j-m_{j+1}+1.$

It will be useful to consider the following ordered set of urns $$B_1,B_2, B_1\cup B_2.$$
In view of the notation to be introduced in the general case we denote these urns as
$$B_{1,1},B_{2,2},B_{1,2}.$$

We will introduce later on an order among certain collections of urns that will yield, in this
particular case,
$$B_{1,1}<B_{2,2}<B_{1,2}.$$

Now perform a total of three consecutive experiments. Each experiment consists of drawing one
ball at random (i.e. with the uniform distribution) from an urn in the ordered set of urns above,
record the outcome as a letter in a word, and continue to the next experiment making sure to return
the ball that has been drawn to its original urn after this experiment has been performed.

The first experiment consists of picking one ball from urn $B_{1,1}=B_1.$
This can give a ball of color $c_1$ or $d_1.$ Record the outcome $c_1$ or $d_1$
as the first letter in a word of three letters, and return the ball to its original urn, $B_{1,1}$.

The second experiment consists of picking one ball from urn $B_{2,2}=B_2.$ This can
result in a ball of color either $c_2$ or $d_2.$ Record the result as the second letter
in a  word that will have a total of three
letters (the colors of the balls chosen in experiments 1,2,3), and return the ball to its original urn, $B_{2,2}$.

The last experiment consists of picking one ball from the union of the urns $B_{1,1}$ and $B_{2,2}$, i.e urn $B_{1,2}.$
The color of the ball in question i.e. $c_1,d_1,c_2$ or $d_2$ is the last letter in our word.
This last ball drawn from $B_{1,2}=B_1\cup B_2$ is then returned to the urn $B_1$ or $B_2$ where it came from.

There is a total of sixteen ($=2 \times 2 \times 4$) possible words that can arise in this fashion from an alphabet of four letters.
These words constitute the set of all possible outcomes of the experiment made up of these three succesive and properly ordered ones.

Since we return the chosen ball at the end of each one of these experiments to its original urn, we have
that the state of the system has not yet changed. This is about to happen now.

We need a rule to decide which of the three values $m_1,m_2,m_3$
will be increased by one unit as the result of our experiment. To
this end we break up the set of sixteen words into three disjoint
and exhaustive sets. These sets will be denoted by $S_{1,3}$,
$S_{2,3}$ and $S_{3,3},$ and the sample space $S_3$ of cardinality
$16$ is given by
$$S_3=\bigcup_{j=1}^{3} S_{j,3}.$$

Each set $S_{j,3}$ consisting of words with three letters will be
obtained by a ``growth process'' starting from the sets we would
have if we had considered the previous case, namely $n=1$, when we
have only one box and we were dealing with $\U(2)$. In that case the
sets are made up of words of one letter, either ${c_1}$ or ${d_1}.$
To make the connection with the general case we will denote these
sets in the case of one urn by $S_{1,2}$ and $S_{2,2},$ and the
sample space by $S_2=S_{1,2}\cup S_{2,2}.$ Explicitly
$S_{1,2}=\{c_1\}$, $S_{2,2}=\{d_1\}$.

\smallskip

Let us come back to the case $n=2$. The class $S_{1,3}$ is formed by
including all three letter words that start as those of $S_{1,2}$
and whose remaining two letters are such that the last one is not
$d_2$, i.e. either $c_1,d_1$ or $c_2.$ Thus
$$S_{1,3}=\{(c_1,c_2,c_1),(c_1,c_2,d_1),(c_1,c_2,c_2),(c_1,d_2,c_1),(c_1,d_2,d_1),(c_1,d_2,c_2)\}.$$

The class $S_{2,3}$ is formed by including all three letter words
that start as those of $S_{2,2}$ and whose remaining two letters are
such that the first one is not $d_2$. Explicitly $S_{2,3}$ is
$$S_{2,3}=\{(d_1,c_2,c_1),(d_1,c_2,d_1),(d_1,c_2,c_2),(d_1,c_2,d_2)\}.$$

Notice that the meaning of the requirement "not $d_2$" is quite
different when it applies to the second urn $B_{2,2}$ as above, or
to the third urn $B_{1,2}$ as in the previous case.

Finally $S_{3,3}$ is obtained by taking the union of all three
letter words that start as in $S_{1,2}$ and have $d_2$ as their last
letter, together with all words that start as in $S_{2,2}$ and have
$d_2$ as the second letter. Notice that $S_{3,3}$ is obtained by
going over all the classes already built, $S_{1,3}$ and $S_{2,3}$,
and replacing the condition not $d_2$ by $d_2.$ The class $S_{3,3}$
is thus made up of two sets of words, namely
\begin{equation*}
\begin{split}
S_{3,3}=&\{(c_1,c_2,d_2),(c_1,d_2,d_2)\}\\
&\cup\{(d_1,d_2,c_1), (d_1,d_2,d_1),(d_1,d_2,c_2),(d_1,d_2,d_2)\}.
\end{split}
\end{equation*}

It takes almost no effort to see that all these $6+4+6=16$ words have been classified into three disjoint and exhaustive classes.

\

Now we compute the total probability of getting a result that belongs to each class.
For the first class $S_{1,3}$ we have,
\begin{equation*}
\frac{(m_1-k_1+1)(m_1-k_2+2)}{(m_1-m_2+1)(m_1-m_3+2)}=a_1^2(\mm,\kk).
\end{equation*}

\noindent For the second class  $S_{2,3}$ we have that the probability is
$$ \frac{(k_1-m_2)(m_2-k_2+1)}{(m_1-m_2+1)(m_2-m_3+1)} =a_2^2(\mm,\kk).$$

\noindent  Finally the total probability of the third class  $S_{3,3}$ is,
\begin{equation*}
\begin{split}
& \frac{ (m_1-k_1+1)(k_2-m_3)}{(m_1-m_2+1)(m_1-m_3+2)}+\frac{(k_1-m_2)(k_2-m_3)}
{(m_1-m_2+1)(m_2-m_3+1)}\\
&=\frac{(k_2-m_2)(k_1-m_3+1)}{(m_1-m_3+2)(m_2-m_3+1)}=a_3^2(\mm,\kk).
\end{split}
\end{equation*}

We are ready to give a rule for changing the state of the system in one unit of
time. A result belonging to the subset $S_{j,3}$, $j=1,2,3$, will
lead to a transition to a new state $\mm+\ee_j$, where $m_j$ is
increased by one. In terms of balls this will be achieved by
removing from each urn containing a ball of color $d_{j-1}$ one of
these balls, and adding to each urn containing a ball of color $c_j$
one ball of this color from the bath. When $j=1$ we do no removal.

\section{An urn model for every $\U(n+1)$}\label{urnmodelUn}

In this section we describe a random mechanism that gives rise to a
Markov chain whose one-step transition matrix is
$$\begin{vmatrix} Y_0&X_0&0&\\0&Y_1&X_1&0&\\&0&Y_2&X_2&0&\\
&&0&Y_3&X_3&0\\&&&\cdot&\cdot&\cdot&\cdot\end{vmatrix},$$ appearing
in the factorization \eqref{factorization} and where the matrices
$X_i,Y_i$ are defined in \eqref{XY}.

A configuration is a set of  $n+1$ values of the integers $m_i$,
$1\le i\le n+1$, subject to the constrains $m_1\ge k_1\ge m_2\ge
\cdots\ge m_n\ge k_n\ge m_{n+1}$ where the integers $k_i$ remain
unchanged throughout time. We will construct a stochastic process
whereby in one unit of time one of the $m_j$ is increased by one
with probability given by
\begin{equation}\label{coefficientsab1}
a_j^2(\mm,\kk)=\left\vert\frac{\prod_{i=1}^{n}(k_i-m_j-i+j-1)}
{\prod_{i\ne j}(m_i-m_j-i+j)}\right\vert.
\end{equation}

Consider $n$ urns $B_1,\dots,B_n$. In urn $B_j$ place $m_j-k_j+1$
balls of color $c_j$ and $k_j-m_{j+1}$ balls of color $d_j$. We
assume that the colors $c_j, d_j$ are all different. Notice that in
urn $B_j$ may be no ball of color $d_j$, and that the total number
of balls in $B_j$ is $m_j-m_{j+1}+1$.

Consider the following ordered set of urns
\begin{equation*}
B_1,B_2,B_1\cup B_2,B_3,B_2\cup B_3,B_1\cup B_2\cup B_3,\dots,B_n,B_{n-1}\cup B_n,\dots,B_1\cup \cdots\cup B_n.
\end{equation*}
The union of urns is an urn whose content is the union of the set of balls in each urn in the union. Observe that the total number of urns under consideration is $n(n+1)/2$. Let
$$B_{k,j}=B_k\cup B_{k+1}\cup\cdots\cup B_j,\quad 1\le k\le j.$$
Clearly $B_{j,j}=B_j$, and the set of all urns
$$\{B_{k,j}: 1\le k\le j\le n\}$$
is ordered lexicographically according to: $(k,j)<(r,s)$ if $j<s$ or if $j=s$ and $r<k$.

We will perform a total of $n(n+1)/2$ consecutive experiments. Each
experiment consists of drawing one ball at random (i.e. with the
uniform distribution) from each urn in the ordered set of urns,
record the outcome as a letter in a word, and continue to the next
experiment making sure to return the ball to the original urn after
this experiment has been performed. One should think of a complete
experiment as consisting of these $n(n+1)/2$ individual experiments.
The transition from the present state of the system to the next one
takes place after the complete experiment is carried out.

The first experiment consists of picking one ball from urn $B_{1,1}$, this can give a ball of color $c_1$ or $d_1$. The result is recorded and the ball is put back in urn $B_{1,1}$. The second experiment consists of picking one ball from  urn $B_{2,2}$, this can result in either  a ball of color $c_2$ or $d_2$.  Record the result as the second letter in a word that will have a total of  $n(n+1)/2$ letters. Put the ball  back in urn $B_{2,2}$. Keep on going by taking successively at random a ball from an urn $B_{k,j}$  and adding the letter corresponding to its color to the right of the word obtained in the previous step. The process finishes once a ball of the last urn $B_{1,n}$ is picked and a final word of $n(n+1)/2$ letters is obtained.

The alphabet is the set $\{c_j,d_j:1\le j\le n\}$ of $2n$ letters. Then the sample space $S_{n+1}$
consists of all words $w$ of $n(n+1)/2$ letters that can be written with such an alphabet with
the restriction that the letters allowed in the place $(k,j)$ correspond to the color of
any ball in urn $B_{k,j}$. The cardinality of the sample space is
$$|S_{n+1}|=\prod_{1\le k\le j\le n}2(j-k+1).$$

 Now by induction on $n\ge 1$ we define a partition of $S_{n+1}$ into $n+1$ disjoint subsets
 $$S_{n+1}=\bigcup_{j=1}^{n+1}S_{j,n+1}.$$

 \bigskip

 For the benefit of the reader the construction will be spelled out in detail for small values
 of $n$ after we describe it in the general case and prove Proposition 5.2.

 \bigskip

 We start with $S_2=S_{1,2}\cup S_{2,2}$ where
 $$S_{1,2}=\{\D_1\}, \quad S_{2,2}=\{d_1\},\quad\D_1=c_1.$$
 Then
 $$|S_{1,2}|=|S_{2,2}|=1,\quad |S_2|=2.$$

 We make the following convention: the symbol $\D_j$ in the $(k,j)$-place of a
 word stands for any color of a ball in urn $B_{k,j}$ different from $d_j$, and the
 letter $x$  in the $(k,j)$-place of a word stands for any possible color of a ball in urn $B_{k,j}$.

If $n\ge2$ we set
$$S_{1,n+1}=\{w_{1,n+1}=w_{1,n}x\cdots x\D_n\in S_{n+1}:  w_{1,n}\in S_{1,n}\}.$$
 Observe that the number of letters in the word $w_{1,n+1}$ to the right of the word $w_{1,n}$ is  $n$.
 Similarly we define
 $$S_{2,n+1}=\{w_{2,n+1}=w_{2,n}x\cdots x\D_nx\in S_{n+1}:  w_{2,n}\in S_{2,n}\}.$$
More generally for $1\le j\le n$ we let
$$S_{j,n+1}=\{w_{j,n+1}=w_{j,n}x\cdots x\D_nx\cdots x\in S_{n+1}:  w_{j,n}\in S_{j,n}\}$$
where the number of $x$'s to the right of $\D_n$ is $j-1$.

The definition of $S_{n+1,n+1}$ is more interesting, namely
\begin{equation*}
\begin{split}
S_{n+1,n+1}=&\{w_{n+1,n+1}=w_{1,n}x\cdots xd_n\in S_{n+1}:  w_{1,n}\in S_{1,n}\}\\
&\cup \{w_{n+1,n+1}=w_{2,n}x\cdots xd_nx\in S_{n+1}:  w_{2,n}\in S_{2,n}\}\\
&\cup\cdots\cup\{w_{n+1,n+1}=w_{n,n}d_nx\cdots x\in S_{n+1}:  w_{n,n}\in S_{n,n}\}.
\end{split}
\end{equation*}

\begin{prop}
Let $n\ge2$. Then for $1\le j\le n$ we have
$$|S_{j,n+1}|=|S_{j,n}|(2(n-j)+1)\prod_{1\le k\le n,\; k\ne j}2(n-k+1),$$
$$|S_{n+1,n+1}|=\sum_{1\le j\le n}|S_{j,n}|\prod_{1\le k\le n,\; k\ne j}2(n-k+1).$$
\end{prop}
\begin{prop} $\{S_{j,n+1}:1\le j\le n+1\}$ is a partition of the
sample space $S_{n+1}$.
\end{prop}
\begin{proof} The proof is by induction on $n\ge1$. For $n=1$ we
have
$$S_{2}=\{\D_1,d_1\}, \quad S_{1,2}=\{\D_1\}, \quad S_{2,2}=\{d_1\}.$$
Thus the statement is true for $n=1$. Now assume that
$S_n=\bigcup_{j=1}^{n}S_{j,n}$ is a partition of $S_n$ for $n\ge1$.
If $w\in S_{n+1}$, then $w=w_{j,n}x\cdots x$ where $w_{j,n}\in
S_{j,n}$ for a unique $j$. The $x$ in the $j$-place of the last $n$
letters is either $d_n$ or of the form $\D_n$. In the first case
$w\in S_{n+1,n+1}$ and in the second case $w\in S_{j,n+1}$. Thus
$S_{n+1}=\bigcup_{j=1}^{n+1}S_{j,n+1}$. At the same time we saw that
$w\in S_{j,n+1}$ for a unique $1\le j\le n+1$. This completes the
proof. \qed
\end{proof}

\

The construction above is now made explicit for small values of $n$.

\noindent 1) $n=2$.
$$S_{1,3}=\{\D_1x\D_2\}, \quad
S_{2,3}=\{d_1\D_2x\}, \quad S_{3,3}=\{\D_1xd_2\}\cup\{d_1d_2x\},$$
$$ |S_{1,3}|=6, \quad |S_{2,3}|=4,\quad |S_{3,3}|=6, \quad |S_3|=16.$$

\

\noindent 2) $n=3$.
$$S_{1,4}=\{\D_1x\D_2xx\D_3\}, \quad S_{2,4}=\{d_1\D_2xx\D_3x\},$$
$$S_{3,4}=\{\D_1xd_2\D_3xx\}\cup\{d_1d_2x\D_3xx\},$$
\begin{equation*}
\begin{split}
S_{4,4}=&\{\D_1x\D_2xxd_3\}\cup\{d_1\D_2xxd_3x\}\\
&\cup\{\D_1xd_2d_3xx\}\cup\{d_1d_2xd_3xx\},
\end{split}
\end{equation*}
$$ |S_{1,4}|=240, \quad |S_{2,4}|=144,\quad
|S_{3,4}|=144, \quad |S_{4,4}|=240, \quad|S_4|=768.$$

\

\noindent 3) $n=4$.
$$S_{1,5}=\{\D_1x\D_2xx\D_3xxx\D_4\}, \quad S_{2,5}=\{d_1\D_2xx\D_3xxx\D_4x\},$$
$$S_{3,5}=\{\D_1xd_2\D_3xxx\D_4xx\}\cup\{d_1d_2x\D_3xxx\D_4xx\},$$
\begin{equation*}
\begin{split}
S_{4,5}=&\{\D_1x\D_2xxd_3\D_4xxx\}\cup\{d_1\D_2xxd_3x\D_4xxx\}\\
&\cup\{\D_1xd_2d_3xx\D_4xxx\}\cup\{d_1d_2xd_3xx\D_4xxx\},
\end{split}
\end{equation*}
\begin{equation*}
\begin{split}
S_{5,5}=&\{\D_1x\D_2xx\D_3xxxd_4\}\cup\{d_1\D_2xx\D_3xxxd_4x\}\\
&\cup\{\D_1xd_2\D_3xxxd_4xx\}\cup\{d_1d_2x\D_3xxxd_4xx\}\\
&\cup\{\D_1x\D_2xxd_3d_4xxx\}\cup\{d_1\D_2xxd_3xd_4xxx\}\\
&\cup\{\D_1xd_2d_3xxd_4xxx\}\cup\{d_1d_2xd_3xxd_4xxx\},
\end{split}
\end{equation*}
$$ |S_{1,5}|=80640, \quad |S_{2,5}|=46080,\quad
|S_{3,5}|=41472, $$
$$|S_{4,5}|=46080, \quad |S_{5,5}|=80640,\quad |S_5|=294912.$$

\

\begin{thm}\label{stochastic}
The probability to obtain a word $w\in S_{j,n+1}$ is
$a_j^2(\mm,\kk)$ for all $1\le j\le n+1$.
\end{thm}
\begin{proof}
Given $(\mm,\kk)$ let $\mm'=(m_1,\dots,m_n)$ and $\kk'=(k_1,\dots,k_{n-1})$. Then from \eqref{coefficientsab1} we get
$$a_j^2(\mm,\kk)=a_j^2(\mm',\kk')\frac{m_j-k_n+n-j+1}{m_j-m_{n+1}+n-j+1},$$
for all $1\le j\le n$. This result allows us to prove the theorem by
induction on $n\ge 1$. When $n=1$ we have only one urn $B_1$ with
$m_1-k_1+1$ balls of color $c_1$ and $k_1-m_2$ balls of color $d_1$.
Thus the probability to obtain a word in $S_{1,2}$ is
$$\frac{m_1-k_1+1}{m_1-m_{2}+1}= a_1^2(\mm,\kk),$$
where $\mm=(m_1,m_2)$ and $\kk=(k_1)$. Similarly the probability to obtain a word in $S_{2,2}$ is
$$\frac{k_1-m_2}{m_1-m_{2}+1}= a_2^2(\mm,\kk).$$
Thus the theorem holds for $n=1$. Now assume that the theorem is true for $n\ge1$. If $1\le j\le n$ we have
$$S_{j,n+1}=\{w_{j,n+1}=w_{j,n}x\cdots x\D_nx\cdots x\in S_{n+1}:  w_{j,n}\in S_{j,n}\}$$
where the number of $x$'s to the right of $\D_n$ is $j-1$. Thus the
probability to obtain a word $w\in S_{j,n+1}$ is equal to
$a_j^2(\mm',\kk')$ times the probability to obtain the symbol $\D_n$
from the urn $B_{j,n}$. Now we recall the composition of urn
$B_{j,n}$. By definition
$$B_{j,n}=B_j\cup B_{j+1}\cup\cdots\cup B_n,$$
the total number of balls $|B_{j,n}|=m_j-m_{n+1}+n-j+1$ and the number of balls of color $d_n$ is $k_n-m_{n+1}$. Therefore the probability to obtain the symbol $\D_n$ from urn $B_{j,n}$ is
$$\frac{m_j-k_n+n-j+1}{m_j-m_{n+1} +n-j+1}.$$

Hence  the probability to obtain a word $w\in S_{j,n+1}$ is
$$a_j^2(\mm',\kk')\frac{m_j-k_n+n-j+1}{m_j-m_{n+1} +n-j+1}=a_j^2(\mm,\kk),$$
which establishes the theorem for all $1\le j\le n$. Since
$\sum_{1\le j\le n+1} a_j^2(\mm,\kk)=1$ (see \eqref{sumauno}) and
$S_{n+1}=\bigcup_{1\le j\le n+1}S_{j,n+1}$ is a partition of
$S_{n+1}$ it follows that the statement of the theorem is also true
for $j=n+1$. \qed
\end{proof}

\

Since we return the chosen ball at the end of each individual experiment to its original urn, we have that the state of the system has not yet changed. This is about to happen now.

The outcome of a complete experiment produces a word that belongs to one of the subsets
$S_{j,n+1}$ in the partition of the sample space $S_{n+1}$. Depending on which subset turns up we take a different action, thus obtaining a random walk in the space of configurations $\mm=(m_1,\dots,m_{n+1})$ which satisfy the constraints $m_1\ge k_1\ge\cdots\ge m_n\ge k_n\ge m_{n+1}$  imposed by the fixed $n$-tuple $\kk=(k_1,\dots,k_n)$. This simple process will give for each configuration $\mm$ a total of  at most $n+1$ possible nearest neighbours  to which we can jump in one transition.

A result belonging to the subset $S_{j,n+1}$, $j=1,\dots,n+1$, will lead to a transition to a new state $\mm+\ee_j$, where $m_j$ is increased by one. In terms of balls this will be achieved by removing from each urn containing a ball of color $d_{j-1}$ one of these balls, and adding to each urn containing a ball of color $c_j$ one ball of this color from the bath.

Notice that all these transitions keep the values of $k_1,\dots,k_n$
unchanged and any transition that would violate the constrains does
not occur because the corresponding probability $a^2_j(\mm,\kk)$ vanishes.

\section{A Young diagram model for $\U(3)$}\label{Young}

To each  configuration $m_1\ge k_1\ge m_2\ge \cdots\ge m_n\ge k_n\ge m_{n+1}\ge0$ we associate  its Young diagram which has $m_1$ boxes in the first row, $k_1$ boxes in the second  row, and so on down to the last row which has $m_{n+1}$ boxes.

\begin{figure}[h]
  \centering
$$D=\begin{array} {l}
\begin{array}{|c|c|c|c|c|c|c|c|}
\hline \makebox[1.5mm]{}&\makebox[1.5mm]{} &\makebox[1.5mm]{}&\makebox[1.5mm]{}&\makebox[1.5mm]{}&\makebox[1.5mm]{}&\makebox[1.5mm]{}&\makebox[1.5mm]{} \\ \hline \end{array}\\
\begin{array}{|c|c|c|c|c|c|} \makebox[1.5mm]{}& \makebox[1.5mm]{} & \makebox[1.5mm]{}&\makebox[1.5mm]{}&\makebox[1.5mm]{}&\makebox[1.5mm]{}\\ \hline \end{array}\\
\begin{array}{|c|c|c|c|c|} \makebox[1.5mm]{}&\makebox[1.5mm]{}&\makebox[1.5mm]{}&\makebox[1.5mm]{}&\makebox[1.5mm]{}\\ \hline \end{array}\\
\begin{array}{|c|c|c|} \makebox[1.5mm]{}&\makebox[1.5mm]{}&\makebox[1.5mm]{}\\ \hline \end{array}\\
\begin{array}{|c|} \makebox[1.5mm]{}\\ \hline \end{array}\\
\end{array}$$
\caption{$\mm=(8,5,1),\;\kk=(6,3)$.}
\end{figure}

\

We will construct a stochastic process whereby in one unit of time
one of the $m_i$ is increased by one with probability
$a_i^2(\mm,\kk)$ see  \eqref{coefficientsab}. As in Section
\ref{urnmodelUn} this will require running some auxiliary
experiments.

We start with the case $n=1$.
We perform the following experiment to decide if we will increase $m_1$ or $m_2$: we choose to insert a box among one of the $m_1-k_1$ last boxes of the first row or to delete a box from the $k_1-m_2$ last boxes of the second row. An insertion can occur either to the left or to the right of one  of the $m_1-k_1$ last boxes. We observe that there are $m_1-k_1+1$ possibilities of an insertion and $k_1-m_2$ possibilities of a deletion. All these are assigned the same probability.

As an output of the experiment we get either a diagram with $m_1+1$ boxes in the first row, or a diagram with $k_1-1$ boxes in the second row. Here we are implicitly assuming that $k_1>m_2$. If $k_1$ were equal to $m_2$ we would get no Young diagram. Thus the sample space $S$ of our auxiliary experiment consists of two (or one) Young diagrams which are obtained from the original one by adding one box to its first row or deleting one from its second row. Let $S_1$ be the subset of $S$ consisting of the diagram with one more box in the first row, and let $S_2$ be the subset of $S$ consisting of the diagram with one less box in the second row (or the empty set). Then the probability to obtain a diagram in $S_1$ after the experiment is performed is
$$\frac{m_1-k_1+1}{m_1-m_2+1}=a_1(\mm,\kk)^2.$$
Similarly the probability to obtain a diagram in $S_2$  is
$$\frac{k_1-m_2}{m_1-m_2+1}=a_2(\mm,\kk)^2,$$
as we wished. In the first case we go from the state $(\mm,\kk)$ to $(\mm+\ee_1,\kk)$, and in the second case we go from the state $(\mm,\kk)$ to $(\mm+\ee_2,\kk)$.

\

Now let us assume that $n=2$. In this case we will perform three
consecutive auxiliary experiments. The first experiment consists of  inserting
a box among one of the $m_1-k_1$ last boxes of the first row or of
deleting a box from the $k_1-m_2$ last boxes of the second row. The
second experiment consists of  inserting a box among one of the
$m_2-k_2$ last boxes of the third row or of deleting a box from the
$k_2-m_3$ last boxes of the fourth row. Finally the third experiment
consists of  inserting or  deleting a box in one of the first four
rows of the diagram as we did in  the previous experiments; odd rows
go along with insertion and even rows with deletion. If $k_1>m_2$
and $k_2>m_3$ the complete experiment gives rise to a triple
$(D_1,D_2,D_3)$ of Young diagrams: $D_1$ is obtained from the
original one by adding one box to its first row or by deleting one
box from the second row, $D_2$ is obtained from the original one by
adding one box to its third row or by deleting one box from the
fourth row, and $D_3$ is obtained by adding one box to the first or
to the third rows of the original diagram or by deleting one box
from the second or the fourth rows.

In what follows we use the following notation: $D$ denotes the Young diagram
corresponding to the original configuration $(\mm,\kk)$ and  $D'=D\pm\ee_j$
denotes, respectively, the diagram obtained from $D$ by adding or deleting
one box to the $j$-row of $D$, $j=1,2,3,4$. Observe that the sample space
consists of all triples of Young diagrams $(D_1,D_2,D_3)$ with
$D_1=D+\ee_1, D-\ee_2$, $D_2=D+\ee_3, D-\ee_4$, and $D_3=D+\ee_1, D-\ee_2, D+\ee_3, D-\ee_4$.

\begin{figure}[h]
\centering
$$\hskip.5cm D+\ee_1=\begin{array} {l}
\begin{array}{|c|c|c|c|c|c|c|c|c|}
\hline \makebox[1.5mm]{}&\makebox[1.5mm]{} &\makebox[1.5mm]{}&\makebox[1.5mm]{}&\makebox[1.5mm]{}&\makebox[1.5mm]{}&\makebox[1.5mm]{+}&\makebox[1.5mm]{}&\makebox[1.5mm]{}  \\ \hline \end{array}\\
\begin{array}{|c|c|c|c|c|c|} \makebox[1.5mm]{}& \makebox[1.5mm]{} & \makebox[1.5mm]{}&\makebox[1.5mm]{}&\makebox[1.5mm]{}&\makebox[1.5mm]{}\\ \hline \end{array}\\
\begin{array}{|c|c|c|c|c|} \makebox[1.5mm]{}&\makebox[1.5mm]{}&\makebox[1.5mm]{}&\makebox[1.5mm]{}&\makebox[1.5mm]{}\\ \hline \end{array}\\
\begin{array}{|c|c|c|} \makebox[1.5mm]{}&\makebox[1.5mm]{}&\makebox[1.5mm]{}\\ \hline \end{array}\\
\begin{array}{|c|} \makebox[1.5mm]{}\\ \hline \end{array}\\
\end{array}$$
\caption{$\mm=(9,5,1),\;\kk=(6,3).$}
\end{figure}
\begin{figure}[h]
\centering
$$D-\ee_2=\begin{array} {l}
\begin{array}{|c|c|c|c|c|c|c|c|}
\hline \makebox[1.5mm]{}&\makebox[1.5mm]{} &\makebox[1.5mm]{}&\makebox[1.5mm]{}&\makebox[1.5mm]{}&\makebox[1.5mm]{}&\makebox[1.5mm]{}&\makebox[1.5mm]{}  \\ \hline \end{array}\\
\begin{array}{|c|c|c|c|c|} \makebox[1.5mm]{}& \makebox[1.5mm]{} & \makebox[1.5mm]{}&\makebox[1.5mm]{}&\makebox[1.5mm]{}\\ \hline \end{array}\\
\begin{array}{|c|c|c|c|c|} \makebox[1.5mm]{}&\makebox[1.5mm]{}&\makebox[1.5mm]{}&\makebox[1.5mm]{}&\makebox[1.5mm]{}\\ \hline \end{array}\\
\begin{array}{|c|c|c|} \makebox[1.5mm]{}&\makebox[1.5mm]{}&\makebox[1.5mm]{}\\ \hline \end{array}\\
\begin{array}{|c|} \makebox[1.5mm]{}\\ \hline \end{array}\\
\end{array}$$
\caption{$\mm=(8,5,1),\;\kk=(5,3).$}
\end{figure}

Thus our sample space $S_3$ has generically $2\times2\times4=16$
elements. The cardinality of $S_3$ can be smaller, for example if
$k_1=m_2$ and $k_2\ne m_3$, then $|S_3|=6$.

Let us partition the sample space $S_3$ into the following three
classes.
\begin{equation}\label{partitionS3}
\begin{split}
S_{1,3}=&\{(D_1,D_2,D_3):D_1=D+\ee_1; D_2=D+\ee_3,D-\ee_4;\\
&\hskip2.6 cm D_3=D+\ee_1,D+\ee_3,D-\ee_2\},\\
S_{2,3}=&\{(D_1,D_2,D_3):D_1=D-\ee_2; D_2=D+\ee_3;\\
&\hskip2.6cm D_3=D+\ee_1, D+\ee_3, D-\ee_2, D-\ee_4\},\\
S_{3,3}=&\{(D_1,D_2,D_3):D_1=D+\ee_1; D_2=D+\ee_3,D-\ee_4; D_3=D-\ee_4\}\\
&\cup\{(D_1,D_2,D_3):D_1=D-\ee_2; D_2=D-\ee_4; \\
&\hskip3.05cm D_3=D+\ee_1, D-\ee_2,D+\ee_3, D-\ee_4\}.
\end{split}
\end{equation}
We have $|S_{1,3}|=6$, $|S_{2,3}|=4$ and $|S_{3,3}|=2+4=6$. By simple inspection we see that $S_3$ is the disjoint union of $S_{1,3}, S_{2,3}$ and $S_{3,3}$.

Then the probability to obtain a diagram in $S_{1,3}$ after a complete experiment is performed is
$$\frac{(m_1-k_1+1)}{(m_1-m_2+1)}\frac{(m_1-k_2+2)}{(m_1-m_3+2)}=a_1^2(\mm,\kk).$$
Similarly the probability to obtain a diagram in $S_{2,3}$  is
$$\frac{(k_1-m_2)}{(m_1-m_2+1)}\frac{(m_2-k_2+1)}{(m_2-m_3+1)}=a_2^2(\mm,\kk).$$
Finally the probability to obtain a diagram in $S_{3,3}$  is
\begin{equation*}
\begin{split}
& \frac{
(m_1-k_1+1)}{(k_2-m_3)}\frac{(m_1-m_2+1)}{(m_1-m_3+2)}+\frac{(k_1-m_2)}{(k_2-m_3)}
\frac{(m_1-m_2+1)}{(m_2-m_3+1)}\\
&=\frac{(k_2-m_2)(k_1-m_3+1)}{(m_1-m_3+2)(m_2-m_3+1)}=a_3^2(\mm,\kk),
\end{split}
\end{equation*}
as desired.

If $k_1=m_2$ and $k_2\ne m_3$ then $|S_{1,3}|=4$, $S_{2,3}=\emptyset$ and $|S_{3,3}|=2$.  The
probability to obtain a diagram in $S_{1,3}$  is
$$\frac{m_1-k_2+2}{m_1-m_3+2}=a_1^2(\mm,\kk).$$
The probability to obtain a diagram in $S_{2,3}$ is $0=a_2^2(\mm,\kk)$, and the probability to obtain a diagram in $S_{3,3}$  is
$$\frac{k_2-m_3}{m_1-m_3+2}=a_3^2(\mm,\kk),$$
as expected.

Now the state of our random walk is modified in one unit of time as follows: if the outcome of the complete experiment above belongs to $S_{j,3}$, then we go from $(\mm,\kk)$ to $(\mm+\ee_j,\kk)$, $j=1,2,3$. In terms of diagrams we move from $D$ to $D+\ee_{2j-1}$, $j=1,2,3$.

\section{A Young diagram model for every $\U(n+1)$}\label{diagrammodelUn}

Given a  Young diagram $D$  corresponding to the original configuration $(\mm,\kk)$,  $D'=D\pm\ee_j$ denotes, respectively, the diagram obtained from $D$ by adding or deleting one box to the $j$-row of $D$, $j=1,\dots,2n+1$.  The stochastic process we are going to construct  will have a transition mechanism determined by first performing a sequence of auxiliary experiments $E_{k,j}$ to be described now. We start by considering the following set of consecutive  pairs of rows of the diagram $D$,
$$\{(1,2),(3,4),\dots,(2n-1,2n)\}.$$

The experiment $E_{k,j}$, $1\le k\le j\le n$, consists of inserting at random a box in an odd row $i$ among  the last $m_i-k_i$ last boxes of such a row, or deleting at random a box in an even row $i$ from   the last $k_i-m_{i+1}$ last boxes of such a row. The row $i$ is also chosen at random in the set of consecutive rows
$$\{2k-1,2k,\dots,2j\}.$$

The sequence of experiments is obtained by ordering them by the lexicographic order $E_{k,j}<E_{r,s}$ if $j<s$ or $j=s$ and $r<k$. Thus our sequence is the following one
$$E_{1,1},E_{2,2},E_{1,2},E_{3,3},E_{2,3},E_{1,3},\dots, E_{n,n},E_{n-1,n},\dots,E_{1,n}.$$

The symbol $D\pm\E_i$ in the place corresponding to the experiment $E_{k,j}$ of an $n(n+1)/2$-tuple of diagrams, will stand for any possible outcome of $E_{k,j}$ except the diagram $D\pm\ee_i$, respectively. While an $X$ in such a place stands for any outcome of $E_{k,j}$. For example in the case $n=2$ considered before, see \eqref{partitionS3}, we can write
\begin{equation*}
\begin{split}
S_{1,3}=&\{(D-\E_2,X,D-\E_4)\},\\
S_{2,3}=&\{(D-\ee_2,D-\E_4,X)\},\\
S_{3,3}=&\{(D-\E_2,X,D-\ee_4\})\cup\{(D-\ee_2,D-\ee_4,X)\}.
\end{split}
\end{equation*}

Now we have a convenient notation to define inductively, for
$n\ge2$,  a  growth process  similar to the one introduced in
Section \ref{urnmodelUn}, to break up the outcomes of the sample
space $S_{n+1}$ into sets $S_{j,n+1}$ $(j=1,\dots,n+1)$ starting
from the partition of $S_n$ into sets $S_{j,n}$ $(j=1,\dots,n)$. Let
$D_{j,n}$ denote any $n$-tuple in the set $S_{j,n}$, then we set
$$S_{1,n+1}=\{D_{1,n+1}=(D_{1,n},X,\cdots, X,D-\E_{2n})\in S_{n+1}:  D_{1,n}\in S_{1,n}\}.$$
 Observe that the number of diagrams in the $(n+1)(n+2)/2$-tuple  $D_{1,n+1}$ to the right of the
 $n(n+1)/2$-tuple $D_{1,n}$ is  $n$.
 Similarly we define
 $$S_{2,n+1}=\{D_{2,n+1}=(D_{2,n},X,\cdots, X,D-\E_{2n},X)\in S_{n+1}:  D_{2,n}\in S_{2,n}\}.$$
More generally for $1\le j\le n$ we let
\begin{equation*}
\begin{split}
S_{j,n+1}&\\
=&\{D_{j,n+1}=(D_{j,n},X,\cdots, X,D-\E_{2n},X,\cdots, X)\in S_{n+1}:  D_{j,n}\in S_{j,n}\}
\end{split}
\end{equation*}
where the number of $X$'s to the right of $D-\E_{2n}$ is $j-1$.

The definition of $S_{n+1,n+1}$ is (as in Section \ref{urnmodelUn})
more interesting, namely
\begin{equation*}
\begin{split}
S&_{n+1,n+1}=\{D_{n+1,n+1}=(D_{1,n},X,\cdots, X,D-\ee_{2n})\in S_{n+1}:  D_{1,n}\in S_{1,n}\}\\
&\quad\cup\{D_{n+1,n+1}=(D_{2,n},X,\cdots, X,D-\ee_{2n},X)\in S_{n+1}:  D_{2,n}\in S_{2,n}\}\\
&\quad\cup\cdots\cup\{D_{n+1,n+1}=(D_{n,n},D-\ee_{2n},X,\cdots, X)\in S_{n+1}:  D_{n,n}\in S_{n,n}\}.
\end{split}
\end{equation*}

Now by induction on $n\ge2$ it is easy to prove that
$\{S_{j,n+1}:1\le j\le n+1\}$ is a partition of $S_{n+1}$. Also by
induction on $n\ge2$ it is possible, as we did to established
Theorem \ref{stochastic}, to prove the following main result.
\begin{thm}
The probability to obtain an $n(n+1)/2$-tuple of diagrams
$D_{j,n+1}\in S_{j,n+1}$  is $a_j^2(\mm,\kk)$ (see
\eqref{coefficientsab1}) for all $1\le j\le n+1$.
\end{thm}

The outcome of a complete experiment produces an $n(n+1)/2$-tuple of
Young diagrams  that belongs to one of the partition subsets
$S_{j,n+1}$ of the sample space $S_{n+1}$. Depending on which subset
turns up we take a different action, thus obtaining a random walk in
the space of configurations $\mm=(m_1,\dots,m_{n+1})$ which satisfy
the constraints
$$m_1\ge k_1\ge\cdots\ge m_n\ge k_n\ge m_{n+1}\ge0,$$
imposed by the fixed $n$-tuple $\kk=(k_1,\dots,k_n)$. This simple process will give for each configuration $\mm$ a total of at most $n+1$ possible nearest neighbours  to which we can jump in one transition.

A result belonging to the subset $S_{j,n+1}$, $j=1,\dots,n+1$, will lead to a transition to a new state $\mm+\ee_j$, where $m_j$ is increased by one.

Notice that all these transitions keep the values of $k_1,\dots,k_n$
unchanged and any transition that would violate the constrains does
not occur because the corresponding probability $a^2_j(\mm,\kk)$ vanishes.

\newcommand\bibit[5]{\bibitem [#1]
{#2}#3, {\em #4,\!\! } #5}

\end{document}